\documentclass{article}
\usepackage[a4paper, total={15cm, 8in}]{geometry}
\usepackage{indentfirst}
\usepackage{graphicx} % Required for inserting images
\usepackage{amsmath,amssymb}
\usepackage{amsthm}
\usepackage{bm}
\usepackage{mathrsfs}
\usepackage{xcolor}
\usepackage{float}
\usepackage[utf8]{inputenc}
\usepackage[export]{adjustbox}
\usepackage{wrapfig}
\usepackage{authblk}
\usepackage{booktabs} 
\usepackage{hyperref}
\usepackage{multirow}
\usepackage{ulem}
\usepackage[title]{appendix}
\usepackage{subcaption}
\theoremstyle{remark}
\newtheorem*{remark}{Remark}
\theoremstyle{definition}
\newtheorem{example}{Example}[section]
\newcommand{\RR}{\mathbb{R}}
\newcommand{\Hh}{\mathcal{H}}
\newcommand{\lb}{\left(}
\newcommand{\rb}{\right)}

\def\XXint#1#2#3{{\setbox0=\hbox{$#1{#2#3}{\int}$}
     \vcenter{\hbox{$#2#3$}}\kern-.5\wd0}}

\newtheorem{theorem}{Theorem}[section]
\newtheorem{lemma}[theorem]{Lemma}
\theoremstyle{definition}
\numberwithin{equation}{section}
\newtheorem{definition}{Definition}
\newtheorem{corollary}[theorem]{Corollary}
\newcommand{\hlap}{(-\Delta)^{\frac{1}{2}}}

\title{A preconditioned boundary value method for advection-diffusion equations with half-Laplacian via spectrum doubling}
\author{P. Yuan$^{1,}$\footnote{\href{mailto:p.yuan@uu.nl}{p.yuan@uu.nl (Correspondence)}}}
\author{P.A. Zegeling$^{1,}$\footnote{\href{mailto:P.A.Zegeling@uu.nl}{P.A.Zegeling@uu.nl}}}
\author{Xian-Ming Gu$^{2,}$\footnote{\href{mailto:guxm@swufe.edu.cn}{guxm@swufe.edu.cn}}}
\affil{{\small $^1$ Mathematical Institute, Utrecht University, Budapestlaan 6, 3584 CD Utrecht, The Netherlands\\
$^2$ School of Mathematics, Southwestern University of Finance and Economics, Chengdu 611130, P.R. China}
}
\date{}

\begin{document}

\maketitle
 \begin{abstract}

In this paper, we study an evolution equation that involves a half-Laplacian operator derived from the Riesz fractional Laplacian, combined with a differential operator \(\mathcal{L}\). Using the identity $(-\Delta)^{1/2}=\mathcal H(\partial_x)$, we introduce a Spectrum Doubling (SD) reformulation that transforms the original half-diffusion equation into a first-order doubled system. The reformulated system exhibits stable and unstable spectral branches, and the original half-diffusion dynamics is recovered on a suitable stable invariant subspace characterized by a compatibility condition on the initial condition. The SD reformulation provides a practical numerical advantage: the half-Laplacian is applied only to the initial condition and source term, avoiding repeated evaluation of singular integrals during time marching. For the resulting integer-order system, we develop a Boundary Value Method (BVM) and study a second-order generalized midpoint scheme. We establish its stability and second-order temporal convergence. The fully discrete scheme leads to a large Kronecker-structured linear system, which is solved efficiently by GMRES with a block $\omega$-circulant preconditioner. Under simultaneous diagonalizability of the spatial discretization matrices, the preconditioner can be implemented efficiently through fast discrete transforms. Numerical experiments for three evelutionary models confirm the theoretical convergence results and demonstrate the robustness and efficiency of the proposed method, including in strongly advective regimes. The experiments also show that the approach remains effective when the Hilbert transform is evaluated numerically, and illustrate the applicability of the SD framework to a nonlocal Schr\"odinger-type example.
\end{abstract}

\quad boundary value method

\vspace{1mm}

%%%%%%%%%%%%%%%%%%

\section{Introduction}
In this paper, we study an evolutionary partial differential equation (PDE) that involves the half-Laplacian operator $\hlap$,  along with an additional differential operator \(\mathcal{L}\). The half-Laplacian, also referred to as the half fractional diffusion operator, can be regarded as a nonlocal generalization of the classical Laplacian. It arises in various physical applications, such as the Peierls-Nabarro model describing dislocations in crystals \cite{crystals}, and in the Benjamin-Ono equation related to hydrodynamics \cite{Benjamin-Ono1, Benjamin-Ono2}. 

More precisely, we consider the initial value problem
\begin{equation}
\left\{
\begin{aligned}
&\partial_t u = -\epsilon (-\Delta)^{\frac12}u + \mathcal L(u) + f(x,t),
\qquad \epsilon>0,\quad (x,t)\in \mathbb R\times[0,T],\\
&u(x,0)=u_0(x),
\end{aligned}
\right.
\label{original_half}
\end{equation}
where we define the half-Laplacian $\hlap$ via the Riesz fractional Laplacian \cite{ten_laplacian}, which is given by the singular integral \cite{ten_laplacian}:
\begin{equation}
\left(-\Delta\right)^{\frac{\alpha}{2}} u(x) = C_1( \alpha)  \text{P.V.} \int_{\mathbb{R}} \frac{u(x)-u(y)}{|x-y|^{1+\alpha}}  dy,~~~\alpha\in(0,2),\label{def_Laplacian}
\end{equation}
with \(C_1(\alpha) = \frac{2^\alpha\Gamma\bigl(\tfrac{1+\alpha}{2}\bigr)}{\sqrt{\pi} \bigl|\Gamma\bigl(-\tfrac{\alpha}{2}\bigr)\bigr|}\) and P.V. indicates the Cauchy principal value of the integral. Throughout this paper, the definition \eqref{def_Laplacian} is used both in the whole-line setting, where $u:\mathbb R\to\mathbb R$, and in the periodic setting, where $u$ is taken to be an $2L$-periodic function on $\mathbb R$. In addition, $\mathcal{L}$ is a spatial differential operator defined as
\[
\mathcal{L} = \delta \partial_x^k, \quad k \in \mathbb{N},
\]
where $\delta$ is a drift or diffusion coefficient. The function \(f(x,t)\) represents a given source term. We focus on the case $k\in\{0,1\}$ in this work. When $k=1$, \eqref{original_half} becomes the Kolmogorov equation with L\'evy flights studied by Caffarelli and Vasseur \cite{caffarelli2010drift}, also known as the fractional Fokker-Planck equation and superimposes a systematic bias or flow onto the jump-driven diffusion, analogous to a prevailing wind carrying a pollutant plume as it disperses \cite{FFP1}. 

Due to the fact that $(-\Delta)^\frac{1}{2}$ is defined via a singular integral on $\RR$, a robust numerical scheme capable of efficiently capturing the behaviour of this nonlocal operator is a challenge. 
In practical numerical schemes, a common way is to truncate the integral \eqref{def_Laplacian}, see \cite{truncation1,truncation2}, where a rapidly decaying function $u$ is often assumed and the case $\alpha=1$ is either excluded or not well-behaved. To avoid truncation, Cuesta et al. \cite{mapping} proposed a coordinate transformation $x = l \cot(s)$ with $l>0$, which maps $\RR$ into the finite interval $s \in [0, \pi]$ and hence avoids truncation of the domain. Other related approaches for approximating the fractional Laplacian, under additional regularity assumptions on $u$, include a nonlocal finite element method that discretizes the weak formulation and assembles a singular double-integral bilinear form \cite{acosta2017fractional}, as well as a rational spectral method that constructs Galerkin and collocation schemes by deriving explicit formulas for the Fourier transforms and fractional Laplacians of the rational basis \cite{tang2020rational}.

In contrast to approaches that compute $\hlap$ directly through singular-integral quadrature, we use the Hilbert transform identity $\hlap=\mathcal H(\partial_x)$ to reformulate the original half-diffusion PDE. This avoids repeatedly evaluating the nonlocal operator during the time-marching process and places the problem into an integer-order framework that is more suitable for our subsequent analysis and numerical discretization. Specifically, we apply $-\epsilon(-\Delta)^{\frac12}$ to both sides of \eqref{original_half}. Using Hilbert transform identity $(-\Delta)^{\frac12}u=\mathcal H(\partial_x u)$,
we obtain
\[
\epsilon^2(-\Delta)^{\frac12}\bigl((-\Delta)^{\frac12}u\bigr)
=
\epsilon^2\mathcal H\partial_x\bigl(\mathcal H\partial_x u\bigr)
=
\epsilon^2\mathcal H^2 \partial_{xx}u
=
-\epsilon^2\Delta u.
\]
Thus, we rewrite the original half-diffusion PDE as a backward wave-type equation, which is also a second-order Cauchy problem \cite{acp_book}, and, equivalently, as a first-order system in $(u,v)$.
We refer to this procedure as \textit{Spectrum Doubling} (SD), because the reformulated operator exhibits two spectral branches, so that each spectral value of the original generator $-\epsilon\hlap+\mathcal{L}$ corresponds to a doubled spectral structure in the new system. We show that this doubled system splits naturally into stable and unstable branches, and that the original half-diffusion dynamics is recovered by restricting the flow to a stable invariant subspace determined by a compatibility condition on the initial condition. On this subspace, the induced semigroup coincides with the contraction semigroup generated by the original operator. Numerically, the half-Laplacian is then applied only to the initial condition $u_0$ and the source term $f$, while the time-marching itself no longer requires repeated evaluation of a singular integral. Regarding the initial condition, the half-Laplacian of some functions can be calculated explicitly via the Hilbert transform, as listed in Appendix \ref{App_a}. For more general functions, fast approximations are discussed in Section \ref{sec2}.

After implementing the SD reformulation, it is allowed to consider applying some existing numerical scheme aimed at computing the integer-order PDE. However, many common time-integration schemes (e.g., certain Linear Multistep methods (LMMs) and Runge–Kutta methods) can fail to maintain stability for this problem due to the backward term $-\Delta u$. We therefore introduce a Boundary Value Methods (BVMs) as an alternative that generalize LMMs by imposing boundary conditions at both the initial and final steps, thus providing unconditional stability and high-order accuracy, which is well suited to the doubled spectral structure of the semi-discrete system. After outlining the basic BVM framework, we focus on a second-order generalized midpoint scheme, establish its eigenvalue-based stability properties, and prove second-order temporal convergence. The resulting fully discrete problem takes the form of a large Kronecker-structured linear system. To solve it efficiently, we employ a block $\omega$-circulant preconditioner within GMRES, so that each preconditioning step reduces to FFT-based transforms together with a set of independent auxiliary systems, making the overall solver naturally parallelizable. Moreover, when the spatial operators are simultaneously diagonalizable, these systems can be further decoupled in an appropriate transformed spatial basis.

In the numerical section, we test the proposed preconditioned BVM on three representative half-diffusion models: a pure half-diffusion equation, a half-diffusion reaction equation, and an advection-dominated half-diffusion equation. The results confirm the expected second-order convergence and show that the method gives accurate solutions in all cases. Compared with the standard BVM, the preconditioned version has clear advantages for large-scale systems while keeping the same accuracy. By using the simultaneous diagonalization of the spatial matrices, the preconditioner can be applied efficiently through fast discrete transforms. Even in the advection-dominated case, where the discrete system is highly ill-conditioned, the method remains stable, accurate, and efficient. Numerical tests also show that the method still performs well when the Hilbert transforms are computed numerically rather than given in closed form. 

The article is organized as follows. Section \ref{sec2} recalls the definition and basic properties of the half-Laplacian, together with its relation to the Hilbert transform. Section \ref{sec3} introduces the SD reformulation, which converts the original half-diffusion equation into a backward wave-type equation, and establishes the theoretical connection between the reformulated problem and the original dynamics. Section \ref{sec4} presents the numerical framework, including the BVM time discretization, its stability and convergence analysis, and the solution of the resulting Kronecker-structured linear systems with suitable preconditioning strategies. Section \ref{sec5} reports numerical experiments for several representative half-diffusion models. Section \ref{sec6} further illustrates the SD framework through a nonlocal Schr\"odinger example. Finally, Section \ref{sec7} concludes the paper and discusses possible directions for future research.

\section{Preliminaries}
\label{sec2}
In this section, we introduce some fundamental properties of operators and notations which will be used in our study.

\subsection{ The half-Laplacian $\hlap$ }
Among the different ways of defining the fractional Laplacian \cite{ten_laplacian}, we use the singular integral definition \eqref{def_Laplacian} by taking $\alpha=1$, i.e.,

\begin{equation}
    -(-\Delta)^{\frac{1}{2}}u(x) 
    = \frac{1}{\pi}   P.V.\int_{-\infty}^{\infty} \frac{u(y) - u(x)}{(x-y)^2}  dy,\label{def_half_lap}
\end{equation}
If we assume that $u$ is differentiable and bounded (i.e., $u \in C^1_b(\mathbb{R})$) and $\lim_{|x|\to\infty}\frac{du}{dx}(x)=0$, then \eqref{def_half_lap} can also be written as
\[\begin{aligned}
        -(-\Delta)^\frac{1}{2}u(x) 
    &=\frac{1}{\pi}   P.V.\int_{-\infty}^{\infty} \frac{u(y) - u(x)}{(x-y)^2}  dy\\
     &= \lim_{y\to\infty}\frac{u(y)-u(x)}{x-y} + \lim_{y\to-\infty}\frac{u(y)-u(x)}{x-y} + \frac{1}{\pi} P.V.\int_{-\infty}^{\infty} \frac{\frac{d u}{d x}(y)}{y - x}  dy\\
     &=\frac{1}{\pi} P.V.\int_{-\infty}^{\infty} \textcolor{red}{\frac{\frac{d u}{d x}(y)}{y - x}}  dy.
\end{aligned}
\]
which shows that the half-Laplacian is the \textit{Hilbert transform} of the derivative of $u(x)$, namely
\[
    -(-\Delta)^\frac{1}{2}u(x)\equiv -\mathcal{H}\!\Bigl(\tfrac{du}{dx}\Bigr)(x).
\]

\subsection{The Hilbert Transform \texorpdfstring{$\mathcal{H}$}{H}}
The Hilbert transform is an important operator in harmonic analysis and the theory of singular integrals. It is used to study the boundedness of operators on $L^p$ spaces and to investigate the regularity of solutions to partial differential equations. For a function $f(x)\in L^p(\mathbb{R})$ with $1 < p < \infty$, it is defined by
\[
    \mathcal{H}(f)(x)
    = \frac{1}{\pi} P.V.\int_{-\infty}^{\infty} \frac{f(y)}{x - y}   dy, \quad x \in \mathbb{R}.
\]
Its key properties include:
\begin{itemize}
    \item \textbf{Linearity:} For functions $f$ and $g$ and scalars $\alpha, \beta \in \mathbb{R}$,
    \[
      \mathcal{H}(\alpha f + \beta g) = \alpha \mathcal{H}(f) + \beta \mathcal{H}(g).
    \]
    \item \textbf{Inversion Property:} The Hilbert transform is a skew-involution,
    \[
      \mathcal{H}^2 = -\mathcal{I},
    \]
    where $\mathcal{I}$ is the identity operator.
    \item \textbf{Differentiation:} The Hilbert transform commutes with differentiation, so
    \[
      \mathcal{H}\lb\frac{d^nf}{dx^n}\rb(x) = \frac{d^n}{dx^n}\mathcal{H}\bigl(f\bigr)(x), \quad n \in \mathbb{N}.
    \]
\end{itemize}

We list some common functions and their Hilbert transforms in Appendix A. For more general functions without an explicit formula, Weideman \cite{compute_hilberrt} used expansions in rational eigenfunctions of the Hilbert transform combined with the Fast Fourier Transform (FFT) to derive the approximation
\begin{equation}
        \mathcal{H}(f)(x)\approx \frac{1}{1-\mathrm{i}x}\sum_{n=-N}^{N-1}\mathrm{i} \mathrm{sgn}(n) a_n \Bigl(\tfrac{1+\mathrm{i}x}{1-\mathrm{i}x}\Bigr)^n,\label{compute_H}
\end{equation}
where
\[
    a_n=\frac{1}{N}\sum_{j=-N+1}^{N-1}\bigl(1-\mathrm{i} \tan\tfrac{1}{2}\theta_j\bigr) f\bigl(\tan\tfrac{1}{2}\theta_j\bigr) \exp\bigl(\mathrm{i}n \theta_j\bigr),
    \quad \theta_j=\frac{j\pi}{N}, \; |j| < N.
\]
Other approaches include the rectangle rule by Kree \cite{compute_hilbert_2} and a piecewise–linear interpolation method proposed by Bilato et al. \cite{compute_hilbert_3}, which yields a simple $O(N \log N)$ algorithm for the Hilbert transform by exploiting an antisymmetric Toeplitz matrix representation and the discrete trigonometric transform. In the special case where $f$ is periodic with period $2L$, the pseudospectral method with FFT can be directly employed to evaluate $\mathcal{H}(f)(x)$ \cite{compute_hilbert_4}, namely:
\begin{equation}
    \mathcal{H}( f)(x) \approx \sum_{k=-N/2}^{N/2-1} -\mathrm{i}\mathrm{sgn}(k)\hat{f}_k\exp\Bigl(\mathrm{i}\frac{k\pi}{L}x\Bigr),\label{compute_H_periodic}
\end{equation}
where $\hat{f}_k = \frac{1}{N}\sum_{j=0}^{N-1} f(x_j) \exp\Bigl(-\mathrm{i}\frac{k\pi}{L}x\Bigr)$ with $x_j=\frac{jL}{N}$.
\section{Spectrum Doubling Reformulation}
\label{sec3}

Let $\mathbb{T}_{2L}:=\RR/(2L\mathbb{Z})$ and denote $D(\mathcal{A})$ as the domain of an operator $\mathcal{A}$.  We work with the half-Laplacian $(-\Delta)^{\frac12}$ defined on 
\[D((-\Delta)^{\frac12})=\left\{\begin{aligned}
   & H^1(\RR),\ \text{for}\ u\in L^2(\RR),\\
   & H^1(\mathbb{T}_{2L}):= L^2(\mathbb{T}_{2L})\cap H^1(\Omega),\ \text{for}\ u\in L^2(\mathbb{T}_{2L}).
\end{aligned}\right.\]
Assume $u_0\in H^2(\RR)\cap D(\mathcal L)$ and
\[
f\in C\big([0,T]; H^1(\RR)\cap D(\mathcal L)\big)\cap C^1\big([0,T]; L^2(\RR)\big).
\]
We also assume $u$ is a classical solution  of \eqref{original_half} such that
\[
u\in C\big([0,T]; H^2(\RR)\cap D(\mathcal L)\big)\cap C^1\big([0,T]; H^1(\RR)\big)\cap C^2\big([0,T]; L^2(\RR)\big).
\]
It gives that $\hlap u\in C([0,T];H^1(\RR))$, and for every $t\in[0,T]$,
\[
u(\cdot,t),\ \widetilde u(\cdot,t)\in D(\hlap)\cap D(\mathcal L),\qquad
\partial_t u(\cdot,t)\in D(\hlap).
\]
Then consider the original problem
\[    \partial_t u -\mathcal{L}(u) -f  = -\epsilon(-\Delta)^{\frac{1}{2}}u,~~u(x,0) = u_0(x).\]
Applying $-\epsilon(-\Delta)^\frac{1}{2}$  to both sides yields
\[-\epsilon\hlap\lb \partial_t   -\mathcal{L} \rb(u) +\epsilon(-\Delta)^{\frac{1}{2}}f = \epsilon^2(-\Delta)^{\frac{1}{2}}\lb(-\Delta)^{\frac{1}{2}}\rb (u).\]
Under the above assumptions, $-\epsilon(-\Delta)^\frac{1}{2}$ commutes with $\partial_t - \mathcal{L}$, as can be seen by expanding the left-hand side:
\[\partial^2_t u=\epsilon^2(-\Delta)^{\frac{1}{2}}\lb(-\Delta)^{\frac{1}{2}}\rb (u) - \mathcal{L}^2(u)+2\mathcal{L}\lb\partial_t u \rb - \epsilon \mathcal{H}\left(\partial_x f\right)- \mathcal{L}(f) + \partial_t f\]
Following the properties of Hilbert transform $\mathcal{H}$ in section 2.2, we have
\begin{equation}
    \partial^2_t u=-\epsilon^2 \Delta u - \mathcal{L}^2(u)+ 2 \mathcal{L}\left(\partial_t u \right) - \mathcal{L}(f) - \epsilon \mathcal{H}(\partial_x f) + \partial_t f,\label{ds_eq}
\end{equation}
with initial conditions
\begin{equation}
    u(x,0) = u_0(x),\ \partial_t u (x,0) = -\epsilon\mathcal{H}(\partial_x u_0)(x) + \mathcal{L} u_0(x) + f(x,0).\label{ds_eq_initial}
\end{equation}

There is a special case occurs  when $\epsilon$ is a pure imaginary coefficient, for example, the Schrödinger-type equation with the half-diffusion:
\begin{equation}
    \mathrm{i}\partial_t u  = (-\Delta)^{\frac{1}{2}}u + \mathrm{i}\mathcal{L}(u) + \mathrm{i}f(x,t).\label{schrodinger_ds}
\end{equation}
Under the similar regularity assumption of $u(x,t)$ and $f(x,t)$, we have its corresponding doubled equation given by
\begin{equation}
    \partial^2_t u= \Delta u - \mathcal{L}^2(u) + 2 \mathcal{L}\left(\partial_t u \right) - \mathcal{L}(f) - \mathrm{i}\mathcal{H}(\partial_x f) + \partial_t f,\label{ds_schrodinger}
\end{equation}
which is reduced to a standard wave equation if $\mathcal{L}=0$ for the homogeneous case $f=0$.

Following the equation \eqref{ds_eq}, we introduce $v = \partial_t u  - f$, then (\ref{ds_eq}) is split in time, and (\ref{original_half}) is reformulated as a first-order system: 
\begin{equation}
   \left\{\begin{aligned} 
\partial_t u &=v + f, \\
\partial_t v&=-\epsilon^2 \Delta u+2 \mathcal{L}(v)- \mathcal{L}^2(u) +  \mathcal{L}(f) - \epsilon \mathcal{H}(\partial_x f),
\end{aligned}\right.\label{ds_2eqs}
\end{equation}
with 
\begin{equation}
    \left\{\begin{aligned}
   u(x,0)&=  u_0(x),\\
   v(x,0) &= -\epsilon\mathcal{H}(\partial_x u_0)(x) + \mathcal{L} u_0(x).
\end{aligned}\right.\label{ds_2eqs_initial}
\end{equation}
The reformulation of \eqref{original_half} into the first‐order system \eqref{ds_2eqs} is referred to as \textit{spectrum doubling} because it symmetrises the spectrum of the evolution operator \(-\epsilon(-\Delta)^{\frac12}\), as will be analysed in section \ref{sec31}. The idea behind this procedure is natural since we do not want the unknown function \( u \) to be involved in the Hilbert transform. Converting the nonlocal PDE into an integer-order system could simplify the analysis and calculations. 
\begin{remark}
To justify the operator manipulations used in deriving the spectrum-doubling equation, we assume the solutions $(u, v)$ possess sufficient regularity. Specifically, we consider $(u, v) \in H^2(\mathbb{R}) \times H^1(\mathbb{R})$, ensuring that all intermediate steps are well-defined in a classical sense. The numerical discretization does not require this high regularity, as all that is required is to have well-defined initial conditions on the grid and to evaluate the forcing term. This accurately represents the mild evolution of the solution in the natural energy space, $H^1(\mathbb R)\times L^2(\mathbb R)$. 
\end{remark}

\subsection{Doubled spectral system and the stable invariant subspace}

\subsubsection{Second-order abstract Cauchy problem}
\label{sec31}
Suppose $X$ is a Banach space. Let $u: [0, T] \to X$, and denote the time derivative by $' = \frac{d}{dt}$. We define $g = f' - \mathcal{L}(f) - \epsilon\mathcal{H}(\partial_x f)$ and $u_1 = u'(0)$. Then equation \eqref{ds_eq} can be rewritten as
\begin{equation}
u''(t) - 2\mathcal{L}u'(t) - (-\epsilon^2\Delta - \mathcal{L}^2)u(t) = g(t). \label{acp}
\end{equation}
This is a second-order \textit{abstract Cauchy problem} (ACP), which is well-posed $[0,T]$ if it has unique solutions $u$ for a dense set of initial conditions $(u_0,u_1)\in D(-\epsilon^2\Delta - \mathcal{L}^2)\times X$ which depend continuously on $(u_0,u_1)$ \cite{cosine_family_general}. We first consider a special case in which $\epsilon$ is purely imaginary, for example

\begin{equation}
        u''(t) -2\mathcal{L}u'(t) - (\Delta - \mathcal{L}^2)u(t) = g(t).\label{acp_schrodinger}
\end{equation}
This is a damped ACP for (\ref{ds_schrodinger}) and has been widely studied, see  \cite{acp_book,acp2_wellposed_b2a,cosine_family_general,theory_ds_cosine_family} with the following classical conclusions. For $\mathcal{L} = 0$, (\ref{acp_schrodinger}) is a standard second order Cauchy problem. Since $\Delta$ is a generator of a cosine function, Arendt et al. \cite{theory_ds_cosine_family} show that there exists a unique mild solution to (\ref{acp_schrodinger}) for $f\in L^2(\RR)$, expressed as:
\begin{equation} u(t) = C(t)u_0 + S(t)u_1 + \int^t_0 S(t-s)g(s)ds,\ t \geq 0, \label{L0_mild} \end{equation}
where $S(t):=\int_0^t C(s)ds$ is a sine function and $C(t)$ is a cosine function defined by
\begin{definition}\cite{theory_ds_cosine_family}
    A one parameter family $\{C(t)\}_{t\in\RR^+}$ on a Banach space $X$ is called a cosine family if 
    \begin{itemize}
        \item \(C(0) = \mathcal{I}\);
        \item \(2C(t)C(s) = C(t+s)+C(t-s)~\text{for}~t\ge s\ge 0\);\
    \end{itemize}
Moreover, if \(C(\cdot)u:\RR^+\to X\) is continuous for each \(u\in X\), then $C(t)$ is strongly continuous.
\end{definition}

For the nontrivial case $\mathcal{L}\ne 0$, we observe that $\mathcal{L}^2 + (\Delta - \mathcal{L}^2) = \Delta$ generates a strongly cosine family, then Lightbourne \cite{cosine_family_general} states that if $\mathcal{L}$ generates a group $T_s(t)$ and $g(t)$ is Lipschitz in $\RR$ (or in $\mathbb{T}_{2L}$ for the periodic case), then (\ref{acp}) has a unique mild solution:
\begin{equation}
    u(t) = T_s(t)[C(t)u_0 +S(t)(u_1-\mathcal{L}u_0 )] +\int^t_0T_s(t-s)S(t-s)g(s)ds,~t\ge0.\label{mild_general}
\end{equation}
Under the regularities of $u$ in previous, this mild solution is in fact classical and satisfies the original equation (\ref{ds_eq}), see \cite{wellposed_acp2,cosine_family_general}.
 
We will continue the pure imaginary case in section 6. In this study, we mainly focus on the existence of solutions to (\ref{acp}) with $\epsilon \in \mathbb{R}^+$ and follow the above ideas and results to construct the explicit form of solutions. Similarly, we set $\mathcal{L}=0$ and begin with the homogeneous case, i.e., the backward wave equation (BWE):
\[u''(t) = -\epsilon^2\Delta u,\]
where we assume that $u\in H^2(\RR)$ so that $\Delta u\in L^2(\mathbb R)$ and the Fourier transform is well-defined. On the real line $\mathbb R$, the Laplacian has continuous spectrum. Its generalized eigenfunctions are the plane waves
\begin{equation}
    e_\xi(x)=e^{i\xi x},\qquad \xi\in\mathbb R,\label{generalized_eig}
\end{equation}
 satisfying $-\Delta e^{i\xi x}=\xi^2 e^{i\xi x}.$
Taking the Fourier transform in $x$ of the BWE gives
\begin{equation}\label{eq:mode_ode}
\widehat u''(\xi,t)=
\epsilon^2\xi^2 \widehat u(\xi,t),
\end{equation}
for each $\xi\in\mathbb R$. Then, the characteristic equation of \eqref{eq:mode_ode} is
\[r^2-\epsilon^2\xi^2=0
\quad\Longrightarrow\quad
r=\pm \epsilon |\xi|.\]
Therefore the general solution for each Fourier mode is
\begin{equation}\label{eq:mode_general}
\widehat u(\xi,t)=A(\xi)e^{\epsilon |\xi|t}+B(\xi)e^{-\epsilon |\xi|t}.
\end{equation}
Taking the Fourier transforms of the initial condition \eqref{ds_eq_initial} and substituting them into \eqref{eq:mode_ode} gives, for $\xi\neq0$,
\[A(\xi)=\frac12\Bigl(\widehat u_0(\xi)+\frac{-\epsilon |\xi|\widehat u_0(\xi)}{\epsilon |\xi|}\Bigr)=0,
\qquad
B(\xi)=\frac12\Bigl(\widehat u_0(\xi)-\frac{-\epsilon |\xi|\widehat u_0(\xi)}{\epsilon |\xi|}\Bigr)=\widehat u_0(\xi).\]
Hence the exponentially growing term $e^{\epsilon|\xi|t}$ vanishes identically and
\[\widehat u(\xi,t)=\widehat u_0(\xi)e^{-\epsilon |\xi|t}.\]
The resulting solution exactly matches the solution to the half-diffusion equation
\begin{equation}
    \partial_t u = -\epsilon\bigl(-\Delta\bigr)^\frac{1}{2} u,\label{half_diffusion_eq}
\end{equation}
i.e.,
\begin{equation}\label{example_sol}
u(x,t)=\frac1{2\pi}\int_{\mathbb R}\widehat u_0(\xi) e^{i\xi x} e^{-\epsilon |\xi|t} d\xi,
\end{equation}
where \(\widehat u_0(\xi)=\int_{\mathbb R}u_0(x)e^{-i\xi x} dx.\)  Actually,  since the operator $-\epsilon\hlap$ is self-adjoint and dissipative on $L^2(\RR)$, hence it generates a contraction $C_0$-semigroup $e^{-t\epsilon\hlap}$ on $L^2(\RR)$, and \eqref{half_diffusion_eq} has the unique mild solution
\begin{equation}
    u(t)=e^{-t\epsilon\hlap}u_0,\qquad t\ge0.\label{sol_mild_ori}
\end{equation}
Moreover, if $u_0\in H^1(\RR)$, then $u\in C([0,\infty);H^1(\RR))\cap C^1([0,\infty);L^2(\RR))$, see \cite{existence_mild}.

\subsubsection{Stable invariant subspace and induced contraction semigroup}
We observed that the solution \eqref{eq:mode_general} includes two components $A(\xi)e^{\epsilon|\xi|t}$ and $B(\xi)e^{-\epsilon|\xi|t}$, therefore, the BWE defined in $H^1(\RR)\times L^2(\RR)$ is not globally well-posed. However, the decaying branch $e^{-\epsilon|\xi|t}$only remains when we introduce the initial condition $u'(0) = -\epsilon\hlap u(0)$, and this will give an invariant subspace in $H^2(\RR)\times H^1(\RR)\subset H^1(\RR)\times L^2(\RR)$. 

Suppose that $u_0\in H^2(\RR)$ and $u\in C([0,T];H^2(\RR))\cap C^1([0,T];H^1(\RR))$ in the following context. We rewrite the BWE as a first order system as a special homogeneous case in \eqref{ds_2eqs}.  By setting $v(t) = u'(t)$, we split the BWE in time as:
\begin{equation}
    \left\{\begin{aligned}
    u'(t) &= v, \\
    v'(t) &= -\epsilon^2\Delta u.
\end{aligned}\right.\label{int_semigp}
\end{equation}
 Let $\mathscr{A}=\begin{pmatrix}
       0 & \mathcal{I}\\
       -\epsilon^2\Delta & 0
    \end{pmatrix}$, we note that the spectrum of $\mathscr A$ lies on the real axis because the dynamics contains both the
``$\epsilon\hlap$'' and ``$-\epsilon\hlap$'' directions, and $\sigma(\epsilon\hlap)=[0,\infty)$. Then, define 
\[q^+ = v - \epsilon\hlap u,\qquad q^- = v + \epsilon\hlap u,\]
we directly have
\[(q^+)' = -\epsilon\hlap q^+,\qquad (q^-)' = \epsilon\hlap q^-.\]
We have two spectral branches $\pm\epsilon|\xi|$ in $H^1(\RR)\times L^2(\RR)$, which reveals the spectrum doubled compare to the original problem \eqref{original_half}.
\begin{definition}
    Define a set corresponding to $q^-$ by
    \[\mathbb{M}_s:=\{(u,v)\in H^{1}(\mathbb R)\times L^{2}(\mathbb R):\ v=-\epsilon\hlap u\}.\]
\end{definition}
We call this is a stable invariant subspace, which is a closed linear subspace in $H^1(\RR)\times L^2(\RR)$.

\begin{lemma}\label{equivalence}
    Let the initial state $(u(0),v(0))=(u_0,v_0)$. Under the regular assumption of $u$, if the initial state $(u_0,v_0)\in\mathbb{M}_s$, then we have $(u,v)\in \mathbb{M}_s$ for $t\in[0,T]$.
\end{lemma}
\begin{proof}
    Taking the Fourier transform on $(q^-)'$ yields $\widehat {q^-} = e^{\epsilon|\xi|t}\widehat{q^-}(0)$. Since $\widehat{q^-}(0)=0$ then $q^-(t)\equiv0$, and thus, $v=-\epsilon\hlap u$ for all $t$.
\end{proof}

From this Lemma, we directly implies the following theorem.

\begin{theorem}
\label{thm:stable_branch_equivalence}

Let $u$ satisfy the BWE on $[0,T]$ with regularity 
\[u\in C([0,T];H^2(\RR))\cap C^1([0,T];H^1(\RR)),\]
and assume the compatibility constraint $v_0=-\epsilon\hlap u_0$. Then $u$ satisfies the half-diffusion equation \eqref{original_half} on $[0,T]$ in $L^2(\RR)$.
Furthermore, define $\Phi(t)\colon\mathbb{M}_s\to\mathbb{M}_s$ by
\[
\Phi(t)(u_0,v_0):=(e^{-t\epsilon\hlap}u_0, -\epsilon\hlap e^{-t\epsilon\hlap}u_0),
\qquad (u_0,v_0)\in\mathbb{M}_s.
\]
Then $\{\Phi(t)\}_{t\ge0}$ is a $C_0$-semigroup on the closed subspace $\mathbb{M}_s\subset H^1(\RR)\times L^2(\RR)$.
\end{theorem}

\begin{proof}

The representation $u(t)= e^{-t\epsilon\hlap} u(0)$ follows from uniqueness for \eqref{original_half}.

Immediate from $u'=-\epsilon\hlap u$ and strong continuity of $e^{-t\epsilon\hlap}$ on $H^1(\RR)$, we have $\Phi(t)$ preserves $\mathbb{M}_s$ by construction.

\end{proof}
\begin{remark}
We obtain similar results under the assumption $u\in H^2(\mathbb{T}_{2L})\cap H^1(\mathbb{T}_{2L})$ but with eigenvalues since $-\Delta$ has the discrete spectrum on $L^2(\mathbb{T}_{2L})$.
\end{remark}

This theorem equivalently shows that the dynamics of BWE restricted to $\mathbb{M}_s$ coincides with the half-diffusion \eqref{original_half}.  That is, $\mathscr{A}$ induce a contraction semigroup on the stable invariant subspace $\mathbb{M}_s$ via the restricted flow $\Phi(t)$.

\subsection{Consistency between (\ref{original_half}) and (\ref{ds_eq})}
Having established the equivalence between the BWE and the half-diffusion formulation, we now extend this connection to the complete system. More precisely, we show that the original problem \eqref{original_half} and the system \eqref{ds_eq} are equivalent, in the sense that they admit the same solution. For simplicity, we present the argument on $\mathbb R$, and the periodic case on $\mathbb T_{2L}$ follows analogously by replacing the Fourier transform with the Fourier series, so that the continuous frequency $\xi\in\mathbb R$ is replaced by the discrete modes $\xi_n=\frac{n\pi}{L}$$,~n\in\mathbb Z$. In this sense, the periodic problem is the discrete spectral analogue of the whole-line problem. The fundamental solution of \eqref{original_half} is given in Appendix~B.

\subsubsection{Explicit Solution of (\ref{ds_eq})}
Following the idea in \cite{acp2_wellposed, acp2_wellposed_b2a_2}, we assume the solution has the form \eqref{mild_general} since the cosine family \( C(t) \), generated by \( -\epsilon^2\Delta \) and subject to the initial conditions (\ref{ds_eq_initial}), does not lead to exponential growth for \( t > 0 \) from (\ref{example_sol}).  Since $-\epsilon^2\Delta$ is a self-adjoint and densely defined closed operator with $D(-\Delta)=H^{2}(\RR)$, it follows that $\mathscr{A}$ is closed and densely defined on $H^{2}(\RR)\times H^{1}(\RR)$ and generates once integrated semigroup $\mathscr{S}(t)$, see \cite{integrated}. To explicitly express $C(t)$, we define $\mathscr{S}(t)$ in a formal way:
    \[\mathscr{S}(t)=\frac{\exp(t\mathscr{A})-\mathscr{I}}{\mathscr{A}},\]
where $\mathscr{I}=\begin{pmatrix}
    \mathcal{I} & 0\\
    0 & \mathcal{I}
\end{pmatrix}$ and $\mathscr{S}$ is understood by continuous extension at the spectral point 0 (e.g. $\frac{\exp(z)-1}{z}\to 1$ as $z\to 0$). Formally separating the series of \(\exp(t\mathscr{A})\) into its even and odd powers:
\[
\begin{aligned}\exp(t\mathscr{A}) =& \sum_{n=0}^\infty \frac{(t\mathscr{A})^n}{n!} 
    =\sum_{k=0}^\infty\frac{t^{2k}}{(2k)!} \mathscr{A}^{2k} +  \sum_{k=0}^\infty\frac{t^{2k+1}}{(2k+1)!} \mathscr{A}^{2k+1}\\
   =&\sum_{k=0}^\infty\frac{t^{2k}}{(2k)!}
    \begin{pmatrix}
    (\sqrt{-\epsilon^2\Delta})^{2k} & 0 \\
    0 & (\sqrt{-\epsilon^2\Delta})^{2k}
    \end{pmatrix} + \sum_{k=0}^\infty\frac{t^{2k+1}}{(2k+1)!}
    \begin{pmatrix}
    0 & (\sqrt{-\epsilon^2\Delta})^{2k} \\
    (\sqrt{-\epsilon^2\Delta})^{2k+2} & 0
    \end{pmatrix}.
\end{aligned}
\]
Using the fact that $\cos(i\beta X) = \sum^\infty_{k =0}\frac{1}{(2k)!}(\beta X)^{2k}$ and $\sin(i\beta X) = i\sum^\infty_{k =0}\frac{1}{(2k+1)!}(\beta X)^{2k+1}$ we conclude that
\[\mathscr{S}(t)=\frac{\exp(t\mathscr{A})-\mathscr{I}}{\mathscr{A}} = \begin{pmatrix}
    \frac{\sin(i\sqrt{-\epsilon^2\Delta}t)}{i\sqrt{-\epsilon^2\Delta}} & \frac{\cos(i\sqrt{-\epsilon^2\Delta}t)-\mathcal{I}}{{-\epsilon^2\Delta}}\\ \cos(i\sqrt{-\epsilon^2\Delta}t)-\mathcal{I} & \frac{\sin(i\sqrt{-\epsilon^2\Delta}t)}{i\sqrt{-\epsilon^2\Delta}}
\end{pmatrix}.\]  
\begin{theorem}\cite{theory_ds_cosine_family,integrated}
    The operator \(\mathscr{A}\) generates a cosine function $C(t)$ on a Banach space \(X\) if and only if \(\begin{pmatrix}
       0 & \mathcal{I}\\
       \mathcal{A} & 0
    \end{pmatrix}\) generates a once integrated semigroup \(\mathscr{S}\) on \(X \times X\). In that case, \(\mathscr{S}\) is given by
\[
\mathscr{S}(t) =
\begin{pmatrix}
S(t) & \int_0^t S(s)   ds \\
C(t) - \mathcal{I} & S(t)
\end{pmatrix},
\]
where \(S(t) = \int_0^t C(s)   ds. \)
\end{theorem}
Therefore, we derive a cosine function $C(t)=\cos(i\sqrt{-\epsilon^2\Delta}t) = \cosh(\sqrt{-\epsilon^2\Delta}t)$ in form generated by $-\epsilon^2\Delta$ defined on $H^2(\RR)$. It can be checked that the solution of the BWE indeed can be represented in the form \eqref{L0_mild}. 
Furthermore, by applying the variation-of-constants formula for cosine families, we write the solution of \eqref{acp} the form of (\ref{mild_general}), provided that \( T_s(t) \) is a group generated by \(\mathcal{L} \). We will then verify that the mild solution to \(\eqref{ds_eq}\), constructed by the semigroup approach, yields the same evolution as the classical solution to \(\eqref{original_half}\).  We will show this coincidence of the two solutions in specific cases, namely $\mathcal{L} = 0$, $\mathcal{L} = \delta \mathcal{I}$, and $\mathcal{L} = \delta \partial_x $, in the following section.

\subsubsection{Equivalence of Solutions}

\paragraph{Case 1: $\mathcal{L}=0$} We have $T_s(t)=\mathcal{I}$ on $L^2(\mathbb{R})$. Using the standard cosine-family formula, the mild solution can be written as
\begin{equation}\label{R_mild_start}
u(t)
=\cos\!\bigl(i\sqrt{-\epsilon^2\Delta} t\bigr)u_0
+\frac{\sin\!\bigl(i\sqrt{-\epsilon^2\Delta} t\bigr)}{i\sqrt{-\epsilon^2\Delta}}u_1
+\int_0^t
\frac{\sin\!\bigl(i\sqrt{-\epsilon^2\Delta}(t-s)\bigr)}{i\sqrt{-\epsilon^2\Delta}} g(s) ds,
\quad t\in[0,T].
\end{equation}
Using the generalized eigenfunctions $e_\xi$, we have
\[
\cosh(t\sqrt{-\epsilon^2\Delta})e_\xi=\cosh(\epsilon|\xi|t)e_\xi,
\qquad
\sinh(t\sqrt{-\epsilon^2\Delta})e_\xi=\sinh(\epsilon|\xi|t)e_\xi,\qquad \xi\in \RR.
\]
Recalling that \(u_1 = -\epsilon \mathcal{H}(\partial_x u_0) + \mathcal{L}(u_0) + f_0\) with \(f_0 = f(x, 0)\)
we rewrite $u_1,g$ in Fourier variables,
\[
\widehat u_1(\xi)=-\epsilon|\xi|\widehat u_0(\xi)+\widehat f_0(\xi) ~\text{ and }~
\widehat g(\xi,t)=\widehat f'(\xi,t)-\epsilon|\xi| \widehat f(\xi,t),
\]
where 
\(\epsilon \mathcal H(\mathcal D e_\xi)=\epsilon|\xi| e_\xi
\) 
for the mode $e_\xi$, see \cite{ten_laplacian,hilbert_book}.
 By applying \eqref{R_mild_start} on each $e_\xi$ and substituting $\widehat u_1,\widehat g$, we obtain
\[\begin{aligned}
\widehat u(\xi,t)
&=\cosh(\epsilon|\xi|t)\widehat u_0(\xi)
+\frac{\sinh(\epsilon|\xi|t)}{\epsilon|\xi|}\widehat u_1(\xi) 
+\int_0^t
\frac{\sinh(\epsilon|\xi|(t-s))}{\epsilon|\xi|}\widehat g(\xi,s) ds\\
&=\widehat u_0(\xi)\Bigl[\cosh(\epsilon|\xi|t)-\sinh(\epsilon|\xi|t)\Bigr]
+\frac{\sinh(\epsilon|\xi|t)}{\epsilon|\xi|}\widehat f_0(\xi)\\
&\quad
+\int_0^t
\frac{\sinh(\epsilon|\xi|(t-s))}{\epsilon|\xi|}
\Bigl[\widehat f'(\xi,s)-\epsilon|\xi| \widehat f(\xi,s)\Bigr] ds \\
&=\widehat u_0(\xi)e^{-\epsilon|\xi|t}
+\int_0^t e^{-\epsilon|\xi|(t-s)}\widehat f(\xi,s) ds
\end{aligned}\]
where we applied the integration by parts at the last step. Therefore, the explicit solution on $\mathbb R$ is given by 
\begin{equation}
    u(x,t)=\frac{1}{2\pi}\int_{\mathbb R}
\left[
\widehat u_0(\xi)e^{-\epsilon|\xi|t}
+\int_0^t e^{-\epsilon|\xi|(t-s)}\widehat f(\xi,s) ds
\right]e^{i\xi x} d\xi.\label{sol_bwe_heat}
\end{equation}
 Since the multiplier $e^{-\epsilon|\xi|t}$ is exactly the Fourier transform of the Poisson kernel
\[
P_a(x):=\frac{1}{\pi}\frac{a}{a^2+x^2},\qquad a>0,
\]
it satisfies
\[
\widehat{P_a}(\xi)=e^{-a|\xi|}.
\]
Hence \(e^{-t\sqrt{-\epsilon^2\Delta}}\) is precisely the Poisson semigroup on \(\mathbb R\), and thus
\begin{equation}
    u(x,t) = \int_{\mathbb R}P_{\epsilon t}(x-y) u_0(y) dy
      +\int_0^t\int_{\mathbb R}P_{\epsilon(t-s)}(x-y) f(y,s) dy ds.\label{Poisson}
\end{equation}

\paragraph{Case 2: $\mathcal{L} = \delta\mathcal{I}$} A similar analysis to Case 1 applies, but with a different semigroup $\{T_s(t)\}_{t\ge0}$ defined by $T_s(t)u:=e^{\delta t}u$. Therefore, based on \eqref{Poisson}, we have
\[    u(x,t) = e^{\delta t}\int_{\mathbb R}P_{\epsilon t}(x-y) u_0(y) dy
      +\int_0^t e^{\delta(t-s)}\int_{\mathbb R}P_{\epsilon(t-s)}(x-y) f(y,s) dy ds.\]
In the Fourier domain, by applying integration by parts to  $\int^t_0 T_s(t-s)S(t-s)\widehat f'(s)ds$ we obtain
\[T_s(t)S(t)\widehat f_0 + \int^t_0 T_s(t-s)S(t-s)\widehat g(s)ds = \int^t_0e^{(\delta-\epsilon|\xi|)(t-s)}\widehat f(s)ds.\] 
This gives 
\begin{equation}\label{sol_bwe_trans}
    u(x,t) = \frac{1}{2\pi}\int_\RR\left[ \widehat u_0(\xi)e^{(\delta-\epsilon|\xi|)t} + \int^t_0e^{(\delta-\epsilon|\xi|)(t-s)}\widehat f(\xi,s)ds\right] e^{i\xi x} d\xi.
\end{equation}

\paragraph{Case 3:  $\mathcal{L} = \delta\partial_x $} We proceed as in the previous cases, but now we work with the shifted semigroup \(\{T_s(t)\}_{t\ge 0}\), which is defined by  
   \[(T_s(t)u)(x) = u(x + \delta t).\]
 Consequently, the solution in the convolution form is given by
\[    u(x,t) = \int_{\mathbb R}P_{\epsilon t}(x+\delta t-y) u_0(y) dy
      +\int_0^t \int_{\mathbb R}P_{\epsilon(t-s)}(x+\delta (t-s)-y) f(y,s) dy ds.\]
In the Fourier domain, $T_s(t)$ acts as the multiplier $e^{i\delta\xi t}$. The solution then is expressed as
\begin{equation}\label{sol_bwe_adv}
    u(x,t) = \frac{1}{2\pi}\int_\RR\left[\widehat u_0(\xi)e^{(i\delta\xi-\epsilon|\xi|)t} + \int_0^t e^{(i\delta\xi - \epsilon|\xi|)(t-s)} \hat{f}(\xi,s) ds\right]e^{i\xi x}d\xi.
\end{equation}

\section{The time-integration method}
\label{sec4}
Regarding the semi-discretized scheme of the first-order system \eqref{ds_2eqs}, most LMMs fail to maintain stability because the eigenvalues are symmetrically distributed on the real axis. To address this, we introduce a class of numerical methods called BVMs, and provide the corresponding stability conditions along with error estimations. 

\subsection{Quadratic eigenvalue Problem}
We denote by $P$ the discretized form of $-\epsilon^2\Delta-\mathcal{L}^2$ and by $Q$ the discretized form of $2\mathcal{L}$. In this section, $\mathcal{L}$ is discretized by a direct $n$-th order central finite-difference stencil, which we denote by $D_n$.

We first consider the whole-line discretization. Let $x_j = jh$ for $j \in \mathbb{Z}$ define an infinite uniform grid with step size $h > 0$, and let $u(t)=\{u_j(t)\}_{j\in\mathbb{Z}}$ be the vector of grid values. Since $D_n$ is a central stencil of finite width, it can be written in the form
\begin{equation}\label{Dn-stencil}
    (D_n u)_j(t)=\frac{1}{h^n}\sum_{k=-r}^{r} c_k u_{j+k}(t),\qquad j\in\mathbb{Z},
\end{equation}
for some half-width $r\ge1$, where the coefficients satisfy $c_{-k}=(-1)^n c_k$ together with the usual consistency conditions for the $n$-th derivative \cite{bvm_book}. Consequently, on the grid, we have  
\[P = -\epsilon^2 D_2 - \delta D_n ~\text{ and }~ Q = 2\delta D_n.\]
By defining $U=[u,v]^\top$ and $G(t)=\left(\{f(x_j,t)\}_{j\in\mathbb Z}, 
\{(\mathcal L(f)-\epsilon \mathcal H(\partial_x f))(x_j,t)\}_{j\in\mathbb Z}\right)^\top$ as the vector of sampled forcing terms, we obtain the following semi-discrete system of ODEs equivalent to \eqref{ds_2eqs}:
\begin{equation}
    U’(t)=\mathbf{D}U(t)+G(t),\qquad 
    \mathbf{D}=
    \begin{bmatrix}
        0 & I\\
        P & Q
    \end{bmatrix},
    \label{ode_original}
\end{equation}
where $\mathbf{D}$ is an operator on bi-infinite grid functions.

Since $D_n$ is translation-invariant on $\mathbb{Z}$, the spectral distribution of $\mathbf{D}$ is determined by the discrete Fourier symbol. For the Fourier mode $e^{ij\theta}$, $\theta\in[-\pi,\pi]$, we define
\begin{equation}\label{symbols}
    d_n(\theta)=\frac{1}{h^n}\sum_{k=-r}^{r} c_k e^{ik\theta}.
\end{equation}
Then the symbols of $P$ and $Q$ are
\(
p(\theta)=-\epsilon^2d_2(\theta)-\delta d_n(\theta)\) and \(q(\theta)=2\delta d_n(\theta),
\)
and the $2\times2$ symbol matrix of $\mathbf{D}$ is
\[
\widehat{\mathbf{D}}(\theta)=
\begin{bmatrix}
0 & 1\\
p(\theta) & q(\theta)
\end{bmatrix}.
\]
Therefore, the corresponding characteristic equation is
\begin{equation}\label{QEP-symbol}
    \lambda^2-q(\theta)\lambda-p(\theta)=0.
\end{equation}
Hence the two spectral branches are
\begin{equation}\label{lambda-branches}
    \lambda_\pm(\theta)
    =
    \frac{q(\theta)\pm\sqrt{q(\theta)^2+4p(\theta)}}{2}
    =
   \delta d_n(\theta)\pm \sqrt{\delta^2d_n(\theta)^2-\delta d_n(\theta)-\epsilon^2d_2(\theta)},
    \qquad \theta\in[-\pi,\pi].
\end{equation}
Accordingly, the whole-line discrete spectrum is described by the two branch curves
\[
\{\lambda_+(\theta):\theta\in[-\pi,\pi]\}
\cup
\{\lambda_-(\theta):\theta\in[-\pi,\pi]\},
\]
which displays the expected "spectrum-doubling" property.
The central symmetry $c_{-k}=(-1)^n c_k$ implies
$d_n(-\theta)=(-1)^n d_n(\theta) $ and $ \overline{d_n(\theta)}=d_n(-\theta).$
Thus the symmetry of the symbol $d_n(\theta)$ dictates the structural topology of the spectrum in the complex plane: (i). If $n$ is even, then $d_{n}(\theta)$ is a real-valued even function of $\theta$. Consequently,
$\delta^2d_{n}(\theta)^2-\delta d_{n}(\theta)-\epsilon^2d_2(\theta)\in\mathbb{R},$
so the whole-line spectrum is symmetric with respect to the real axis.
(ii). If $n$ is odd, then $d_{n}(\theta)$ is purely imaginary and odd in $\theta$, it follows that $d_n(-\theta) = \overline{d_n(\theta)} = -d_n(\theta)$.  Therefore, the whole-line discrete spectrum is again symmetric with respect to the real axis. Geometrically, the pair of roots $\lambda_\pm(\theta)$ typically traces out two origin-symmetric curves in the complex plane, like a twisted loops ($\infty$-shaped) that closes at the origin as $\theta$ varies continuously over $[-\pi, \pi]$. 
\paragraph{Periodic boundary conditions as a special case.}
The periodic case is obtained by restricting the infinite grid to one period and identifying indices modulo $m$. Then $D_n$ becomes a circulant matrix, so the finite-dimensional block matrix $\mathbf{D}$ is exactly diagonalized by the discrete Fourier basis. Its eigenvalues are therefore obtained by sampling the whole-line branch curves \eqref{lambda-branches} at the discrete Fourier frequencies $\theta_k=\frac{2\pi k}{m}, k=0,1,\dots,m-1.$
That is,
\begin{equation}\label{periodic-eigs}
    \lambda_{k,\pm}
    =
    \delta d_n(\theta_k)\pm\sqrt{\delta^2d_n(\theta_k)^2-\delta d_n(\theta_k)-\epsilon^2d_2(\theta_k)},
    \qquad k=0,1,\dots,m-1.
\end{equation}
Hence the periodic problem is a finite-dimensional exact sampling of the whole-line discrete spectral distribution. 

\paragraph{Large-domain Dirichlet truncation for $u(\cdot,t)\in C_0(\mathbb{R})$.}
In actual computations, one cannot work directly on the infinite grid. Besides extended-periodic approximation, another standard approach is to truncate $\mathbb R$ to a sufficiently large bounded interval and impose homogeneous Dirichlet conditions at the artificial boundaries. 
Such a large-domain Dirichlet truncation should be understood as a finite-dimensional approximation of the original problem on $\mathbb R$.

Let $\Omega=(0,L)$ be a large computational interval with mesh
\(
\overline{\Omega}:0=x_0\le \cdots \le x_m=L.
\)
Using the same central stencil as in \eqref{Dn-stencil}, together with homogeneous Dirichlet closure at the endpoints, we obtain a $2(m-1)\times 2(m-1)$ block matrix $\mathbf D$. This truncation is justified by the locality of the stencil. If the half-width of $D_n$ is $r$, then at a node $x_j$ the discrete operators involve only the values $u_{j-r},\dots,u_{j+r}$. Hence, at every interior node whose stencil is entirely contained in $\Omega$, the truncated discrete operator coincides exactly with the whole-line discrete operator. Moreover, if $u(\cdot,t)\in C_0(\mathbb R)$, then the solution values near the artificial boundaries become arbitrarily small as $L\to\infty$. Therefore, over a fixed finite time interval, the large-domain Dirichlet problem provides a consistent approximation of the whole-line dynamics on the interior region. At the spectral level, the Dirichlet-truncated problem has only finitely many eigenvalues. Since the interior stencil is translation-invariant and the boundary modification affects only finitely many rows, the truncated matrix can be viewed as a finite-section, boundary-perturbed approximation of the whole-line operator. Hence, for large $L$, its eigenvalues are expected to reflect the same leading spectral geometry as the whole-line branches \eqref{lambda-branches}, up to boundary effects and possible outliers. 
\begin{remark}
    If the sampled solution is compactly supported, then on that active region the truncated and whole-line schemes act identically. In this sense, the interior numerical evolution is an exact inheritance of the whole-line discretization. Under the stronger assumption $u(\cdot,t)\in\mathcal S(\mathbb R)$, the approximation is even more accurate because of the rapid decay at infinity.
\end{remark}

As discussed above, the periodic problem provides an exact finite-dimensional sampling of the whole-line spectral branches, whereas for sufficiently large computational domains the Dirichlet-truncated matrix inherits the same leading spectral pattern up to boundary effects. To make this connection explicit, we now consider the simplest finite-dimensional example. It already captures the essential difficulty for time integration: the semi-discrete spectrum is not contained in the left half-plane, and therefore standard IVP solvers based on classical absolute stability regions do not provide unconditional stability.

\begin{example}
    Taking $\epsilon = 1$ with a central discretization for $\Delta$ and setting $\mathcal{L}= 0$, this yields
\[\mathbf{D} = \left(\begin{matrix}
    0 & I  \\
    P & 0
\end{matrix}\right),\ P= -\frac{1}{h^2}\begin{pmatrix}
-2 & 1 & \phantom{0} & \phantom{0} & \phantom{0} \\
1 & -2 & 1 & \phantom{0} & \phantom{0} \\
\phantom{0} & \ddots & \ddots & \ddots & \phantom{0} \\
\phantom{0} & \phantom{0} & 1 & -2 & 1 \\
\phantom{0} & \phantom{0} & \phantom{0} & 1 & -2
\end{pmatrix}, \]
Denote the eigenvalues of $P$ by $\lambda_i^P$. The eigenvalues of $\mathbf{D}$ can be found to be $\lambda_i=\pm\sqrt{\lambda_i^P},\ i = 1,2,...,m$, which are symmetrically located on the real axis. A similar conclusion holds for the periodic discretization, where $P$ is replaced by the corresponding circulant discrete negative Laplacian. Commonly used methods, such as the Backward Differentiation Formula and Runge-Kutta methods, do not yield unconditional stability for this problem \cite{bvm_book}, see Figure \ref{stability_region}. Therefore, we introduce an alternative integration method, that is the \textit{Boundary Value method}. 

\begin{figure}[htbp]
    \centering
    \includegraphics[width=0.4\linewidth]{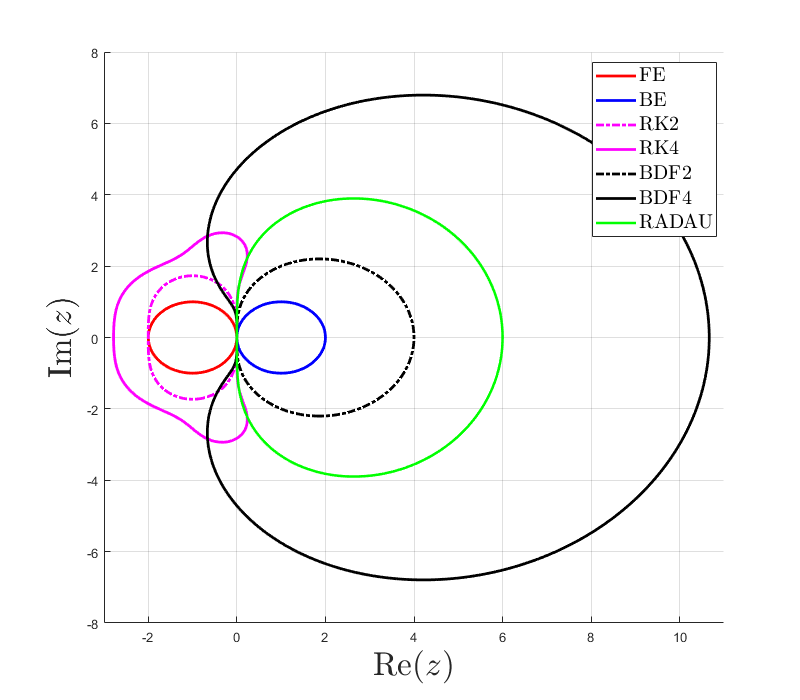}
    \caption{Stability (boundary) locus of the method of explicit$\slash$implicit Euler, RK2, RK4, BDF2, BDF4 and RADAU IIA.}
    \label{stability_region}
\end{figure}
\end{example}

\subsection{Boundary Value method}
BVMs belong to a class of ordinary differential equation (ODE) solvers which can be interpreted as a generalization of LMMs \cite{bvm_book,brugnano1998boundary}. Compared to traditional methods for solving initial value problems (IVPs), BVMs offer the advantages of unconditional stability and high-order accuracy.
Consider the following IVP:
\begin{equation}
    u'(t)=f(u(t),t), \quad u(0)=u_0, \quad t\in[0,T] .\label{ivp}
\end{equation}
\noindent
After discretizing (\ref{ivp}) with a $k$-step LMM, the $k$-conditions are needed to obtain the discrete solution. In BVM, these conditions are imposed to initial and final values of the boundaries, i.e. in this method the continuous IVP (\ref{ivp}) is approximated by means of a discrete boundary value problem. To illustrate this, we discretize the time interval \([0, T]\) into \(N\) equal steps with step size \(\tau = \frac{T}{N}\), defining discrete time points \(u_r \approx u(t_r), t_r=r \tau\) and $f_r=f( u_r,t_r)$, for \(n = 0,1,\dots,N\). Then, the $k$-step formulation of a BVM used to approximate (\ref{ivp}) can be expressed as:
\begin{equation}
   \sum_{r=0}^k \alpha_r u_{r+j}= \tau \sum_{r=0}^k \beta_r f_{r+j}, \quad j=0,1, \ldots, N-k,\label{LMM}
\end{equation}
 The BVM (\ref{LMM}) requires $k_1$ initial conditions and $k_2 = k-k_1$ final conditions, i.e., we need the values of $u_0, u_1, \ldots, u_{k_1-1}$ and $u_{N-k_2+1}, u_{N-k_2+2}, \ldots, u_N$. The initial condition in (\ref{ivp}) provides the value $u_0$. The extra $k_1-1$ initial and $k_2$ final conditions are of the form
\begin{equation}
    \sum_{r=0}^k \alpha_r^{(j)} u_r= \tau \sum_{r=0}^k \beta_r^{(j)} f_r, \quad j=0,1, \ldots, k_1-1,\label{bvm_initial}
\end{equation}
and
\begin{equation}
    \sum_{r=0}^k \alpha_r^{(j)} u_{N-k+r}= \tau \sum_{r=0}^k \beta_r^{(j)} f_{N-k+r} \quad j=N-k_2+1, \ldots, N,\label{bvm_final}
\end{equation}
where the coefficients $\alpha_r^{(j)}$ and $\beta_r^{(j)}$  are chosen such that the truncation errors for the initial and final conditions maintain the same order as that of the primary formula (\ref{LMM}).

Now we derive $N$ equations (\ref{LMM})-(\ref{bvm_final}), which can be represented as 
\[A_e u_e= \tau B_e f_e\left(t_e, u_e\right),\]
where $t_e, u_e \in \RR^{N+1}, A_e, B_e \in \RR^{N \times(N+1)}$ and $f_e=$ $\left(f_0, f_1, \ldots, f_N\right)^\top$.  By partitioning $A_e=\left[a_0, A\right]$ and $B_e=\left[b_0, B\right]$, we isolate the initial condition, allowing us to rewrite the system as: $A_e u_e= \tau B_e f_e\left(t_e, u_e\right)$ as a system for the unknown $u \in \RR^N$ and get
\begin{equation}
    A u= \tau B f(u,t)+F,\label{mat_ivp}
\end{equation}
where $F=-a_0 u_0+ \tau b_0 f( u_0,t_0)$ contains the initial condition. 

In this paper, we consider the following linear system of ODEs, given by:
\begin{equation}
    u'(t)=D u(t)+f(t), \quad u(0)=u_0, \quad t\in[0,T],\label{nonsingular_ode}
\end{equation}
where $u(t)=\left[u_1(t), u_2(t), \ldots, u_m(t)\right]^\top, f(t)=\left[f_1(t), f_2(t), \ldots, f_m(t)\right]^\top$, $D$ is a $m \times m$ matrix, then the BVM for (\ref{nonsingular_ode}) using (\ref{mat_ivp}) can be written as
\begin{equation}
    \left(A \otimes I_m- \tau B \otimes D\right) u= \tau\left(B \otimes I_m\right) f+ \tau \left(b_0 \otimes\left(D u_0+f_0\right)\right)-a_0 \otimes u_0.\label{kron_ode_source}
\end{equation}
where $f_0=f\left(t_0\right)$ and
\[
\begin{aligned}
& u \approx\left[u_1\left(t_1\right), u_2\left(t_1\right), \ldots, u_m\left(t_1\right), u_1\left(t_2\right), u_2\left(t_2\right), \ldots, u_m\left(t_2\right), \ldots, u_1\left(t_N\right), u_2\left(t_N\right), \ldots, u_m\left(t_N\right)\right]^\top, \\
& f=\left[f_1\left(t_1\right), f_2\left(t_1\right), \ldots, f_m\left(t_1\right), f_1\left(t_2\right), f_2\left(t_2\right), \ldots, f_m\left(t_2\right), \ldots, f_1\left(t_N\right), f_2\left(t_N\right), \ldots, f_m\left(t_N\right)\right]^\top.
\end{aligned}
\]

A second-order BVM approximation of (\ref{ivp}) is considered in this paper, which can be obtained by taking $k=2,k_1=1$ \cite{bvm_book}. Since there is only one initial condition and has been given, we choose the midpoint formula as the main method and select another LMM for the final step. We call this method as generalized midpoint method (GMM) which is a symmetric method that has been thoroughly studied in \cite{bvm_book, zegeling_midpoint}. For the initial value problem (\ref{ivp}), the main formulation of the GMM is given by
\begin{equation}
    \frac{1}{2}(u_{j+1}-u_{j-1})= \tau f_{j},\quad j=1,\dots,N-1.\label{mid_main}
\end{equation}
For the choice of the final-step formula, it is common practice to employ a scheme of the same order as the main method. Nevertheless, such a choice is not mandatory. In the BVM setting, the last-point equation affects only one endpoint and therefore enters the discrete problem as a localized perturbation. Provided that the resulting scheme satisfies the required stability conditions, the perturbation generated by a lower-order terminal formula does not compromise the global accuracy of the principal method, see \cite{bvm_book,zegeling_midpoint}. For this reason, we use Backward Euler as the last-point formula:
\begin{equation}
    u_{N}-u_{N-1}= \tau f_{N}.\label{mid_final}
\end{equation}
\noindent
To represent the above second order GMM in matrix form, as in (\ref{mat_ivp}), we specify \( A \), \( B \), \( a_0 \), and \( b_0 \) as follows:
\begin{equation}
    \begin{aligned}
 A=\left(\begin{array}{ccccc}
0 & 1/2 &  & & \\
-1/2 & 0 & 1/2 &  & \\
 & \ddots & \ddots &  \ddots & \\
 & & -1/2 &0 & 1/2 \\
 & &  & -1& 1
\end{array}\right)_{N\times N}
\end{aligned},\ a_0=\left[-\frac{1}{2},0, 0, \ldots, 0\right]^T,\label{A_a}
\end{equation}
\begin{equation}
    B=I_N,\  b_0=0_{N\times1} .\label{B_b}
\end{equation}
\subsubsection{Stability properties}
\begin{definition}
    We say that a polynomial $\rho(z)$ of degree $k = k_1 + k_2$ is a $S_{k_1,k_2}$-polynomial if its roots are such that 
    \[|z_1| \le |z_2| \le \cdots \le |z_{k_1}| < 1 < |z_{k_1+1}|\le |z_k|,\]
    whereas it is a $N_{k_1,k_2}$-polynomial if
    \[|z_1| \le |z_2| \le \cdots \le |z_{k_1}| \le 1 < |z_{k_1+1}|\le |z_k|,\]
    with simple roots of unit modulus.
\end{definition}
By introducing the polynomial
\[\rho(z) = \sum_{j=0}^k\alpha_j z^j,\ \sigma(z) = \sum_{j=0}^k\beta_j z^j,\]
and the shift operator $\mathscr{E}$
\[\mathscr{E} y_n = y_{n+1},\]
We can write (\ref{LMM}) as 
\begin{equation}
    \rho(\mathscr{E})u_j - \tau\sigma(\mathscr{E})f_j = 0.\label{charc}
\end{equation}
Suppose that $f(t, u)$ is sufficiently smooth, by substituting the values of the solution $u(t)$ and expanding (\ref{charc}) at $t = t_j$, we derive that the truncation error is at least $O(\tau^2)$, given the necessary \textit{consistency conditions} \cite{bvm_book} $\rho(1) = 0, \rho'(1) =\sigma(1)$.  Thus, the second-order GMM is clearly consistent.
\begin{definition}
    For a given $q \in \mathbb{C}$, a BVM with $(k_1, k_2)$-boundary conditions is $(k_1, k_2)$-absolutely
stable if $\pi(z, q)= \rho(z)-q\sigma(z)$ is a $S_{k_1, k_2}$~-polynomial. And the region 
\[R_{k_1,k_2} = \{q\in\mathbb{C}: \pi(z,q) \text{ is }S_{k_1,k_2}\},\]
is called $(k_1,k_2)$-absolute region of the $k$-step BVM with the boundary locus, denoted as $\Gamma$, given by
\[\Gamma = \{q(e^{\textbf{i}\theta})\in\mathbb{C}:\theta\in[0,2\pi)\},\text{ with }q(z) = \frac{\rho(z)}{\sigma(z)}.\]
\end{definition}
Then we have the following theorem for the stability of the second order GMM (\ref{mid_main}-\ref{mid_final}):

\begin{theorem}
    (Absolute Stability of the 2nd order GMM) Consider the two-step midpoint method for (\ref{ivp}) given by the difference formula (\ref{mid_main}). This defines a two-step linear multistep method with $(k_1,k_2)=(1,1)$. Furthermore, the $(1,1)$–absolute stability region of this method is
\[R_{1,1} \;=\; \{q\in\mathbb{C} : q \notin [-\mathrm{i},~\mathrm{i}]\}.\]
\end{theorem}
\begin{proof}
    Explicitly, $\pi(z,q) = z^2 - 2q z - 1$ is a quadratic in $z$ whose roots $z=z(q)$ determine the amplification factors of the numerical solution. Solving $\pi(z,q)=0$ gives: 
$$z \;=\; \frac{2q \pm \sqrt{4q^2 + 4}}{2} \;=\; q \pm \sqrt{ q^2 + 1 } . $$
Observe that the product of the two roots is $|z_1| |z_2| = |z_1 z_2| = |-1| = 1$ is a constant. This shows that for each fixed \( q \), exactly one root lies strictly inside the unit circle and the other strictly outside, unless both roots have modulus 1, i.e., two roots lie on the unit circle \( |z| = 1 \). To determine those borderline values of \( q \), we set \( z = e^{i\theta} \) in the polynomial \( \pi(z, q) \). Substituting \( z = e^{i\theta} \) into \( \pi(z, q) = 0 \) then yields:
\[e^{2i\theta} - 1 - 2q e^{i\theta} = 0.\]
Solving for $q$ gives 
\[q \;=\; \frac{e^{2i\theta} - 1}{2 e^{i\theta}} \;=\; \frac{e^{i\theta} - e^{-i\theta}}{2} \;=\; i \sin\theta .\]
This shows that the boundary locus $\Gamma=[\mathrm{-i},~\mathrm{i}]$ as $\theta$ ranges over $[0,2\pi)$.
Equivalently, we have shown:  $\pi(z,q)$ has a root with $|z|=1$ if and only if $q$ lies on the line segment $\Gamma$. Therefore, if $q\notin\Gamma$, because $z_1 z_2 = -1$ we indeed have that $\pi(z,q)$ is an $S_{1,1}$-polynomial, and then the method is $(1,1)$–absolutely stable with
\[\displaystyle R_{1,1} \;=\; \{ q \in \mathbb{C} : q \notin [-i, i] \} . \]
\end{proof}

\subsubsection{Error estimation}
Let $\Omega=(0,L)$ be a sufficiently large computational domain. For solutions that are not compactly supported, we introduce a lifting so that the transformed problem satisfies the homogeneous boundary condition $u|_{\partial\Omega}=0$. Applying the second-order GMM to equation \eqref{ds_eq} then yields the following semi-discrete numerical scheme:

For \( n = 1, 2, \dots, N-1 \) (midpoint method):
\begin{equation}
    \left\{
\begin{aligned}
\frac{u^{n+1} - u^{n-1}}{2 \tau} &= v^n + f^n, \\
\frac{v^{n+1} - v^{n-1}}{2 \tau} &= (-\epsilon^2\Delta - \mathcal{L}^2) u^n + 2\mathcal{L} v^n + \mathcal{L}(f^n) - \epsilon \mathcal{H}(\partial_x f^n),
\end{aligned}
\right.\label{numeric_1}
\end{equation}

At \( n = N \) (backward Euler method):
\begin{equation}
    \left\{
\begin{aligned}
\frac{u^{N} - u^{N-1}}{\tau} &= v^{N} + f^{N}, \\
\frac{v^{N} - v^{N-1}}{\tau} &= (-\epsilon^2\Delta - \mathcal{L}^2)u^{N} + 2\mathcal{L} v^{N} + \mathcal{L}(f^N) - \epsilon \mathcal{H}(\partial_x f^N),
\end{aligned}
\right.\label{numeric_2}
\end{equation}

The initial values \((u^{0},v^{0})=(u_{0}(x), v_{0}(x))\) are prescribed with  
\(v_{0}(x)=-\epsilon\mathcal{H}\!\bigl(\partial_x u_{0}\bigr)(x)+\mathcal{L}u_{0}(x)\).
A standard Taylor expansion shows that the interior stencil~\eqref{numeric_1} possesses a local truncation error \(O(\tau^{3})\), whereas the single backward step~\eqref{numeric_2} is \(O(\tau^{2})\) (see, e.g.,\ \cite{bvm_book}).  
Setting \(\mathcal{P}:=-\epsilon^{2}\Delta-\mathcal{L}^{2}\) and defining the nodal errors
\(e_{u}^{ n}=u(t^{n})-u^{ n}\) and \(e_{v}^{ n}=v(t^{n})-v^{ n}\), we obtain
\begin{equation}
\left\{
\begin{aligned}
e_{u}^{ n+1}-e_{u}^{ n-1} &= 2\tau e_{v}^{ n}+r_{u}^{ n},\\
e_{v}^{ n+1}-e_{v}^{ n-1} &= 2\tau\bigl[\mathcal{P}e_{u}^{ n}+2\mathcal{L}e_{v}^{ n}\bigr]+r_{v}^{ n},
\end{aligned}\quad 1\le n\le N-1\right.
\label{error-int}
\end{equation}
together with
\begin{equation}
e_{u}^{ N}-e_{u}^{ N-1}= \tau e_{v}^{ N}+r_{u}^{ N}, 
\qquad
e_{v}^{ N}-e_{v}^{ N-1}= \tau\bigl[\mathcal{P}e_{u}^{ N}+2\mathcal{L}e_{v}^{ N}\bigr]+r_{v}^{ N},
\label{error-end}
\end{equation}
where \(\|r_{u}^{ n}\|+\|r_{v}^{ n}\|\le C\tau^{3}\) for \(1\le n\le N-1\) and  
\(\|r_{u}^{ N}\|+\|r_{v}^{ N}\|\le C\tau^{2}\).  
Homogeneous Dirichlet boundary conditions are inherited by the error functions,
\(e_{u}^{ n}|_{\partial\Omega}=e_{v}^{ n}|_{\partial\Omega}=0\).

\begin{theorem}[Second-order convergence]
Let $\{(u^{n},v^{n})\}_{n=0}^{N}$ be the solution produced by
\eqref{numeric_1}–\eqref{numeric_2} with time step $\tau=T/N$ and
applied with $(u_{0}(x), v_{0}(x))$.
Then, for sufficiently small $\tau$, the scheme is second-order
convergent, i.e., there exists a constant $C>0$, independent of $N$ and
$\tau$, such that at the final time $t_{N}=T$
\[
\|e_u^{N}\|^{2}+\|e_v^{N}\|^{2}\le C \tau^{4}.
\]
\end{theorem}

\begin{proof}
Denote the norm $\|\phi\|$ in a Hilbert space defined by $\|\phi\|^2=\langle\phi,\phi\rangle=\int_\Omega \phi^2 dx$. For \( n = 1, 2, \dots, N-1 \), multiplying the first and third equations of system \eqref{error-int} by \( e_u^n \) and \( e_v^n \) respectively, and integrating over $\Omega$, we obtain:
\[
\begin{aligned}
\langle e_u^{n+1} - e_u^{n-1}, e_u^n\rangle &= 2 \tau \langle e_v^n, e_u^n\rangle + \langle r_u^n, e_u^n\rangle, \\
\langle e_v^{n+1} - e_v^{n-1}, e_v^n) &= 2 \tau \left[\langle\mathcal{P}(e_u^n), e_v^n\rangle + \langle2\mathcal{L}(e_v^n), e_v^n\rangle\right] + \langle r_v^n, e_v^n\rangle.
\end{aligned}
\]

At the final step \eqref{error-end}, multiplying the first and second equations of the same system by \( e_u^N \) and \( e_v^N \) respectively, and integrating over $\Omega$, yields:
\[
\begin{aligned}
\langle e_u^{N} - e_u^{N-1}, e_u^{N}\rangle &= \tau \langle e_v^{N},e^N_u\rangle + \langle r_u^N, e_u^{N}\rangle, \\
\langle e_v^{N} - e_v^{N-1}, e_v^{N}\rangle &= \tau \langle\mathcal{P}(e_u^{N}) + 2\mathcal{L}\langle e_v^{N}),e^N_v\rangle + \langle r_v^N, e_v^{N}\rangle.
\end{aligned}
\]

Then, summing the equations from \( n = 1 \) to \( N - 1 \) and including the terms corresponding to \( n = N \), we obtain:
\begin{equation}
    \langle e_u^{N}, e_u^{N}\rangle - \langle e_u^{0}, e_u^{0}\rangle
    = 2 \tau \sum_{n=1}^{N-1} \langle e_v^n, e_u^n\rangle 
    + \tau \langle e_v^{N}, e_u^{N}\rangle 
    + 2 \sum_{n=1}^{N-1}\langle r_u^n, e_u^n\rangle 
    + \langle r_u^N, e_u^{N}\rangle.\label{error1}
\end{equation}

Similarly,
\begin{equation}
    \begin{aligned}
    \langle e_v^{N}, e_v^{N}\rangle - \langle e_v^{0}, e_v^{0}\rangle 
    &= 2 \tau \sum_{n=1}^{N-1} \left( \langle\mathcal{P}(e_u^n), e_v^n\rangle + 2\langle\mathcal{L}(e_v^n), e_v^n\rangle\right) 
    + \tau \langle\mathcal{P}(e_u^{N}), e_v^{N}\rangle + 2\tau \langle\mathcal{L}(e_v^{N}), e_v^{N}\rangle \\
    &\quad + 2 \sum_{n=1}^{N-1}\langle r_v^n, e_v^n\rangle 
    +  \langle r_v^N, e_v^{N}\rangle.
    \end{aligned}\label{error2}
\end{equation}

Given the initial conditions \( e_u^{0} = 0 \) and \( e_v^{0} = 0 \), the left-hand sides of equations (\ref{error1}) and (\ref{error2}) simplify to \(\| e_u^{N} \|^2\) and \( \| e_v^{N} \|^2\), respectively.

On the right-hand side, applying the Cauchy–Schwarz inequality, we have
\[
\langle e_v^n, e_u^n\rangle \leq \frac{1}{2} \left( \| e_u^n \|^2 + \| e_v^n \|^2 \right).
\]

Consequently, we derive the inequality
\begin{equation}
    2 \tau \sum_{n=1}^{N-1} \langle e_v^n, e_u^n\rangle \leq \tau \sum_{n=1}^{N-1} \left( \| e_u^n \|^2 + \| e_v^n \|^2 \right).\label{ieq}
\end{equation}
Similar for \(\langle e_v^{N}, e_u^{N}\rangle\).

For terms involving \( (\mathcal{P}(e_u^n), e_v^n) +(2\mathcal{L}(e^n_v),e^n_v)\), we use the fact that \( e^n_v=-\epsilon(-\Delta)^{\frac{1}{2}}e^n_u + \mathcal{L}e^n_u\) from (\ref{ds_eq}) and derive
\[
(\mathcal{P}(e_u^n), e_v^n) +(2\mathcal{L}(e^n_v),e^n_v) \]
\[= ((-\epsilon^2\Delta-\mathcal{L}^2) e^n_u,(-\epsilon(-\Delta)^{\frac{1}{2}} + \mathcal{L})e^n_u) + 2(\mathcal{L}(-\epsilon(-\Delta)^{\frac{1}{2}} + \mathcal{L})e^n_u,(-\epsilon(-\Delta)^{\frac{1}{2}} + \mathcal{L})e^n_u)
\]
\[= \underbrace{-\epsilon^3\langle-\Delta e^n_u,(-\Delta)^\frac{1}{2}e^n_u\rangle}_{\text{[I]}} +
 \underbrace{\langle \mathcal{L}^2,(\epsilon(-\Delta)^{\frac{1}{2}} - \mathcal{L})e^n_u \rangle + \epsilon^2\langle-\Delta e^n_u, \mathcal{L}e^n_u\rangle+ 2\langle\mathcal{L}(-\epsilon(-\Delta)^{\frac{1}{2}} + \mathcal{L})e^n_u,(-\epsilon(-\Delta)^{\frac{1}{2}} + \mathcal{L})e^n_u\rangle}_{\text{[II]}}.\]
For the part [I], since we have 
\[(-\Delta)^\frac{1}{2}(-\Delta)^\frac{1}{2}u = \mathcal{H}(\partial_x (\mathcal{H}(\partial_x )))u = \mathcal{H}^2(\partial_x ^2)u = -\Delta u \]
for $u, ~(-\Delta)^\frac{1}{2}u\in D((-\Delta)^\frac{1}{2})$, then we conclude that
\[[\text{I}]=-\epsilon^3(-\Delta e^n_u,(-\Delta)^\frac{1}{2}e^n_u)=-\epsilon^3 \mathcal{B}((-\Delta)^\frac{1}{2}e^n_u,(-\Delta)^\frac{1}{2}e^n_u)\le0,\]
where we used the definition of $-(\Delta)^\frac{1}{2}$ in the weak sense via the bilinear form \cite{ten_laplacian}
\begin{equation}
    \langle (-\Delta)^{\frac{1}{2}} u, v \rangle=\mathcal{B}(u, v) = \frac{1}{\pi} \int_{\mathbb{R}} \int_{\mathbb{R}} \frac{(u(x) - u(y))(v(x) - v(y))}{|x - y|^2}  dx dy, \quad \forall v \in H^1_0(\Omega).
\end{equation}  
For the part [II] involving the operator $\mathcal{L}$ we have
\[\langle \mathcal{L}^2,(\epsilon(-\Delta)^{\frac{1}{2}} - \mathcal{L})e^n_u \rangle + \epsilon^2(-\Delta e^n_u, \mathcal{L}e^n_u)+ 2(\mathcal{L}(-\epsilon(-\Delta)^{\frac{1}{2}} + \mathcal{L})e^n_u,(-\epsilon(-\Delta)^{\frac{1}{2}} + \mathcal{L})e^n_u)\]
\[\begin{aligned}
    =&-\epsilon^2\langle \Delta e^n_u,\mathcal{L}e^n_u \rangle + 2\epsilon^2\langle \mathcal{L}(-\Delta)^\frac{1}{2}e^n_u,(-\Delta)^\frac{1}{2}e^n_u \rangle -2\epsilon\langle \mathcal{L}(-\Delta)^\frac{1}{2}e^n_u,\mathcal{L}e^n_u \rangle \\
    &- 2\epsilon\langle\mathcal{L}^2e^n_u,(-\Delta)^\frac{1}{2}e^n_u\rangle + 2\langle\mathcal{L}^2e^n_u,\mathcal{L}e^n_u\rangle + \epsilon\langle\mathcal{L}^2e^n_u,(-\Delta)^\frac{1}{2}e^n_u\rangle- \langle \mathcal{L}^2e^n_u,\mathcal{L}e^n_u\rangle
\end{aligned}\]
\[\le -\epsilon^2\langle \Delta e^n_u,\mathcal{L}e^n_u \rangle+ 2\epsilon^2\langle \mathcal{L}(-\Delta)^\frac{1}{2}e^n_u,(-\Delta)^\frac{1}{2}e^n_u \rangle-\epsilon\mathcal{B}(\mathcal{L}e^n_u,\mathcal{L}e^n_u) + \langle \mathcal{L}^2e^n_u,\mathcal{L}e^n_u\rangle.\]
It is clearly that $[\text{II}]$ vanished if $\mathcal{L}=0$. As for the case $\mathcal{L} = -\delta\mathcal{I}$, we conclude that
\[
    [\text{II}]\le -3\epsilon^2\delta\| \partial_x e^n_u\|^2 -\epsilon\delta^2\mathcal{B}(e^n_u,e^n_u)-\delta^3\|e^n_u\|^2
    \le -\left(\frac{3\epsilon^2\delta\pi}{L}+\delta^3\right)\|e^n_u\|^2-\epsilon\delta^2\mathcal{B}(e^n_u,e^n_u),\]
where Poincáre inequality and $\langle \partial_x (-\Delta)^\frac{1}{2}e^n_u,(-\Delta)^\frac{1}{2}e^n_u \rangle = \|e^n_u\|^2$ are used, see \cite{hilbert_book},\cite{evans}.
For the case \(\mathcal{L} = \delta\partial_x \), we perform a similar analysis under the assumption that there is a stable flux across the boundary, i.e., \(\partial_x u \equiv \text{constant}\) to ensure [II] is bounded. In what follows, we simply consider the homogeneous case where \(\partial_x u \in H^1_0(\Omega)\): 
\[\text{[II]}\le -\epsilon\mathcal{B}(e^n_u,e^n_u)\le0.\]
Moreover, 
\[\text{[I]}=-\epsilon^3 \mathcal{B}((-\Delta)^\frac{1}{2}e^n_u,(-\Delta)^\frac{1}{2}e^n_u) = -\epsilon^3\mathcal{B}(\partial_x e^n_u,\partial_x e^n_u)\le0.\]
Therefore, summing (\ref{error1}) and (\ref{error2}), we obtain the following rough estimate:
\[
\| e_u^{N} \|^2 + \| e_v^{N} \|^2 \leq \tau \sum_{n=1}^{N} \left( \| e_u^n \|^2 + \| e_v^n \|^2 \right) + 2  \sum_{n=1}^{N-1}(\langle r^n_u,e^n_u\rangle + \langle r^n_v,e^n_v\rangle) + \langle r^N_u,e^N_u\rangle+ \langle r^N_v,e^N_v\rangle.
\]

Applying inequality (\ref{ieq}) once again, we arrive at the following estimate:
\[
\begin{aligned}
    \lb\frac{1}{2}-\tau\rb E^{N} \leq &2\tau \sum_{n=1}^{N-1} E^n +  \frac{1}{\tau}\sum_{n=1}^{N-1} \left( \|r_u^n\|^2 + \|r_v^n\|^2 \right) + \frac{1}{2}(\|r_u^N\|^2 + \|r_v^N\|^2)\\
    \le &2\tau\sum_{n=1}^{N-1} E^n + c\tau^4 
\end{aligned}
\]
where $c\in \RR^+$ is a constant and the total error $E^n$ is defined as \(E^n = \| e_u^n \|^2 + \| e_v^n \|^2 \). Assume a relative small time-step $\tau$, for example $\tau<\frac{1}{4}$, then we use the discrete Grönwall inequality conclude that
\[E^N\le 4\tau \sum_{n=1}^{N-1} E^n + c\tau^4\le c\tau^4 e^{4T}.\] 
\end{proof}
It is clear that the numerical scheme has second-order convergence. In practical applications, as suggested by L.~Brugnano \cite{bvm_book}, it is preferable to choose a final step with the same truncation order as the main scheme for stiff problems, to prevent the error from the final step from propagating through the system and degrading the overall convergence order.

\subsection{Parallel preconditioning}
Usually the resulting linear system (\ref{kron_ode_source}) is large and ill-conditioned, and solving it is a core problem in the application of BVMs. If a direct method is employed to solve such a linear system, the operation cost can be very high for practical application. Therefore interest has been turned to iterative solvers, such as the GMRES method. As we know that a clustered spectrum often translates in rapid convergence of GMRES method \cite{Saad86}, so we use the GMRES method for solving the resulting linear system (\ref{kron_ode_source}). To accelerate the convergence, we introduce the block $\omega$-circulant preconditioner
\begin{equation}
\mathbf{P} = \omega(A)\otimes I_m - \tau I_N\otimes D,
\end{equation}
where $\omega(A)$ is the $\omega$-circulant approximation of $A$ introduced for linear multistep methods in boundary value form; see \cite{Bertaccini03}. When $\omega = \exp(\mathrm{i}\theta)$ with $\theta\in(-\pi,\pi]$, the matrix $\omega(A)$ admits the decomposition
\begin{equation}
\omega(A) = \Theta^{*}F^{*}\Lambda_\omega F\Theta,
\end{equation}
where
\[
\Theta = {\rm diag}\left(1,\omega^{-1/N},\ldots,\omega^{-(N-1)/N}\right),
\]
$\Lambda_\omega$ is the diagonal matrix containing the eigenvalues of $\omega(A)$, and $F$ is the discrete Fourier matrix. By the properties of the Kronecker product, the preconditioner can be factorized as
\begin{equation*}
\mathbf{P}
=
\left(\Theta^{*}F^{*}\otimes I_m\right)
\left(\Lambda_\omega\otimes I_m - \tau I_N\otimes D\right)
\left(F\Theta\otimes I_m\right).
\end{equation*}
Therefore, the application ${\bm z} = \mathbf{P}^{-1}\widetilde{{\bm r}}$ required in preconditioned GMRES can be carried out in three steps:
\begin{itemize}
    \item[(a)] Compute
    \(
    \widetilde{{\bm v}}_1
    =
    \left(\Theta^{*}F^{*}\otimes I_m\right)^{-1}\widetilde{{\bm r}},
    \)
    which only involves diagonal scaling and FFTs;

    \item[(b)] Solve the block diagonal system
    \(
    \left(\Lambda_\omega\otimes I_m - \tau I_N\otimes D\right)\widetilde{{\bm v}}_2
    =
    \widetilde{{\bm v}}_1,
    \)
    which is equivalent to solving the $N$ auxiliary systems
    \begin{equation}
    \left(\lambda_j^\omega I_m - \tau D\right)\widetilde{{\bm v}}_{2,j}
    =
    \widetilde{{\bm v}}_{1,j},
    \qquad
    j=1,2,\ldots,N,
    \label{sec4_19}
    \end{equation}
    where
    \(
    \Lambda_\omega = \operatorname{diag}\left(\lambda_1^\omega,\ldots,\lambda_N^\omega\right).
    \)
    
    These systems are independent and can therefore be solved in parallel. For example, if the spatial matrices $P$ and $Q$ in \eqref{ode_original} are simultaneously diagonalizable, then the same auxiliary systems admit a fast transform-based implementation, see Appendix C.

    \item[(c)] Compute
    \(
    \widetilde{{\bm z}}
    =
    \left(F\Theta\otimes I_m\right)^{-1}\widetilde{{\bm v}}_2,
    \)
    which again only involves inverse FFTs and diagonal scaling.
\end{itemize}

\begin{remark}
    The transform-based acceleration described in Appendix C does not introduce a new preconditioner, it is only a fast implementation of the same block $\omega$-circulant preconditioner $\mathbf{P}$.
\end{remark}

The proposed $\omega$-circulant preconditioner can be implemented efficiently in parallel. Its application consists of FFT-based temporal transforms and the solution of $N$ independent shifted systems \((\lambda_j^\omega I_m-\tau D)\). Furthermore, if the spatial operators are simultaneously diagonalizable, each shifted system can be further decoupled in the transformed spatial basis. We also note that the present approach is different from other parallel-in-time strategies, including iterative methods such as Parareal, MGRIT, and PFASST, as well as direct diagonalization-based all-at-once solvers such as that of Liu et al. \cite{liu_pint}. Compare to these methods, the present method follows an all-at-once formulation and uses GMRES together with a block $\omega$-circulant preconditioner for the global space-time system. In this sense, it is more closely related to all-at-once and ParaDiag-type solvers \cite{pint_book}, while retaining the flexibility of an iterative preconditioning framework.

From the algebraic point of view, classical results on block circulant preconditioners for
BVM-type systems show that the preconditioned matrix can be interpreted as a low-rank perturbation of the identity. This structure implies that the spectrum of the preconditioned matrix is clustered around $1\in\mathbb C$, which explains the significant improvement in GMRES convergence when $\mathbf P$ is used. Moreover, under the additional assumption that $\mathbf P^{-1}\mathcal M$ is diagonalizable, one obtains in exact arithmetic the bound that preconditioned GMRES terminates in at most $2mk+1$ iterations, see \cite{pre_1}. Since matrices of the form $I+K+E$ with $\operatorname{rank}(K)\ll n$ and $\|E\|$ small have recently been analysed in detail in \cite{Carr23}, that framework may also be used to further investigate the convergence behaviour of GMRES for the present problem. We leave this analysis for future work.

%%%%%%%%%%%%%

\section{Numerical examples}
\label{sec5}
In this section, we present numerical experiments for three representative nonlocal PDE models to illustrate the performance of the proposed GMM time integrator. The models considered are: (i) a half-diffusion equation, (ii) a half-diffusion reaction equation, and (iii) an advection-dominated model with the half-diffusion. Unless stated otherwise, all simulations are carried out in one spatial dimension on the domain $\Omega=(0,20)$ with uniform grid spacing $h$.
Spatial operators are discretised with second-order central differences.
For case (iii) periodic boundary conditions are imposed, giving the circulant matrices
\[-\Delta \;\approx\; -\frac{1}{h^{2}}
\begin{pmatrix}
-2 & 1 & &  & 1\\
1 & -2 & 1 &  & \\
& \ddots& \ddots & \ddots &  \\
 & & 1 & -2 & 1\\
1 &  & & 1 & -2
\end{pmatrix},
\qquad
\partial_x  \;\approx\; \frac{1}{2h}
\begin{pmatrix}
0 & 1 & & & -1\\
-1 & 0& 1 & & \\
 & \ddots& \ddots & \ddots &  \\
 & &-1 & 0 & 1 \\
1 & & & -1& 0 
\end{pmatrix}.\]

Time integration is performed over the interval $[0,20]$ by the second-order GMM \eqref{A_a}--\eqref{B_b} presented in Section~\ref{sec4}. A circulant preconditioner is used for the larger linear systems, especially in the advection-dominated case, to improve efficiency. Moreover, in all three models the spatial matrices $P$ and $Q$ are simultaneously diagonalizable, which allows the preconditioning operator to be applied efficiently in transform space: by DST-based acceleration in the first two cases and by FFT-based acceleration in the periodic case. 

\subsection{The half-diffusion equation ($\mathcal{L} = 0$)}
We first consider the pure half-diffusion model, in which transport is driven solely by the half-Laplacian:
\begin{equation}
   \left\{ \begin{aligned}
        &\partial_t u= -\epsilon(-\Delta)^{\frac{1}{2}}u+f(x,t),~~(x,t)\in\RR\times(0,T],\\
        &u(x,0) = u_0(x).
    \end{aligned}\right.\label{heat}
\end{equation}

Here $\epsilon>0$ is the anomalous diffusivity and $\hlap$ captures long-range jumps or sub-diffusive spreading typical of Lévy-flight processes, highly heterogeneous media, or visco-elastic trapping.  The source term $f(x,t)$ represents external injection or removal of mass. We choose the diffusion coefficient $\epsilon = 0.1$ and test the homogeneous case with the initial condition $u_0(x) = \frac{1}{(1+x^2)^2}$, the numerical diffusion process is presented in Figure \ref{heat_sol}. By choosing $f(x,t) = -\epsilon\cos(t)  \frac{  x^4 + 6x^2 - 3 }{2 \left( x^2 + 1 \right)^3} - \frac{\sin(t)}{\left( x^2 + 1 \right)^2}$ with the same initial condition $u_0(x)$ and same $\epsilon$ we have the exact solution $u(x,t) = \frac{\cos(t)}{(1+x^2)^2}$. In Figure \ref{converge_0}, we compare the convergence behaviour of the GMM using the direct solver and the preconditioned GMRES solver. The numerical results show that both solvers recover essentially the same convergence order. For the large linear systems considered here, the preconditioned GMRES solver is significantly faster than the direct solver. To quantify this computational advantage, we define the speedup factor by
\[
\mathrm{speedup}=\frac{t_{\mathrm{std}}}{t_{\mathrm{pre}}},
\]
where \(t_{\mathrm{std}}\) and \(t_{\mathrm{pre}}\) denote the CPU times of the standard BVM and the preconditioned BVM, respectively.  Moreover, for some test cases, GMRES yields higher numerical accuracy, which is attributed to its minimal-residual property and the relatively large condition number of the linear system. In most tests, the preconditioned GMRES converges in about 18 iterations at a relative tolerance of $10^{-5}$. To isolate the effect of the proposed preconditioner, we also compared the preconditioned GMRES solver with the corresponding unpreconditioned GMRES solver for this test case. The unpreconditioned GMRES converges much more slowly, with average iteration counts exceeding 200. As shown in Figure~\ref{cpu_compare}, the proposed preconditioner leads to a substantial reduction in CPU time.

\begin{figure}[htbp]
    \centering
    \begin{minipage}{0.45\linewidth}
         \includegraphics[width=\linewidth]{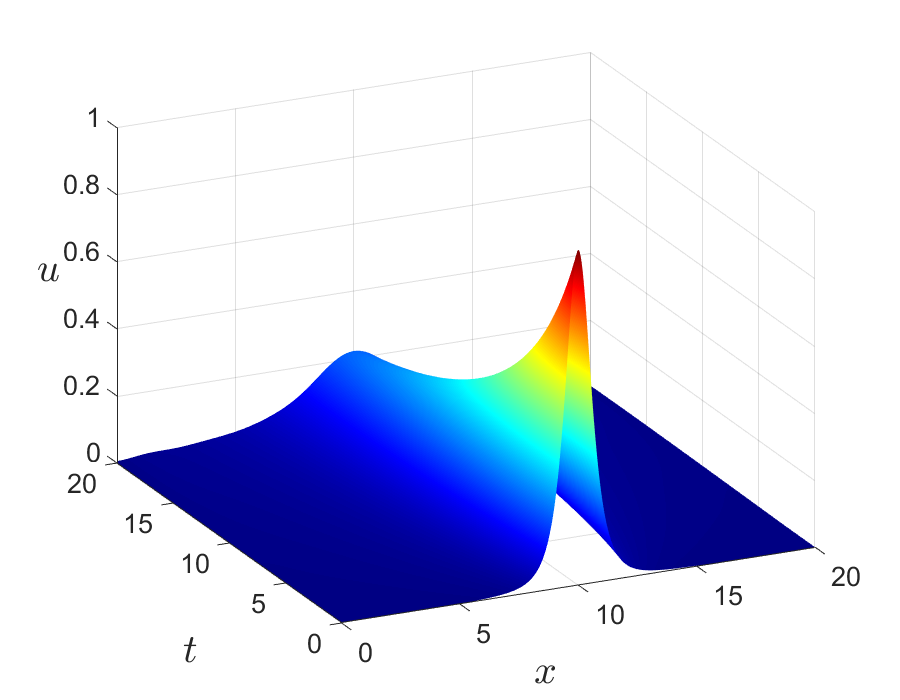}
    \end{minipage}
        \begin{minipage}{0.45\linewidth}
         \includegraphics[width=\linewidth]{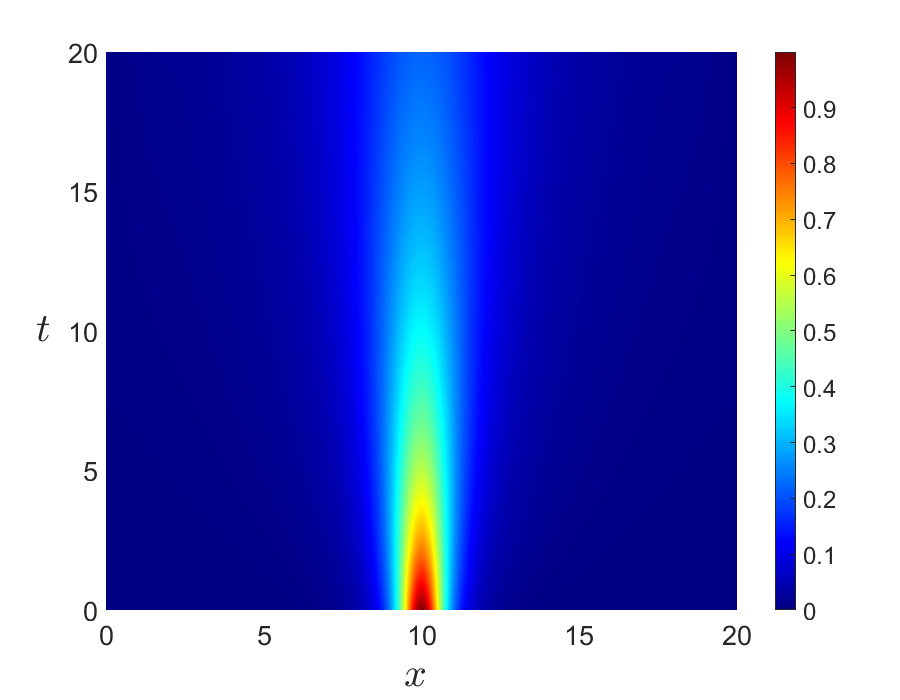}
    \end{minipage}
   
    \caption{Numerical solution of the half-diffusion heat model (\ref{heat}) with $f(x,t) = 0$, initialized with $u_0(x) = \frac{1}{(1+x^2)^2}$ and homogeneous Dirichlet boundary condition, with $\tau=0.0391$ and $h=0.0195$ at time $T=20$.}
    \label{heat_sol}
\end{figure}

\begin{figure}[t]
    \centering
    \begin{minipage}{0.45\linewidth}
        \includegraphics[width=\linewidth]{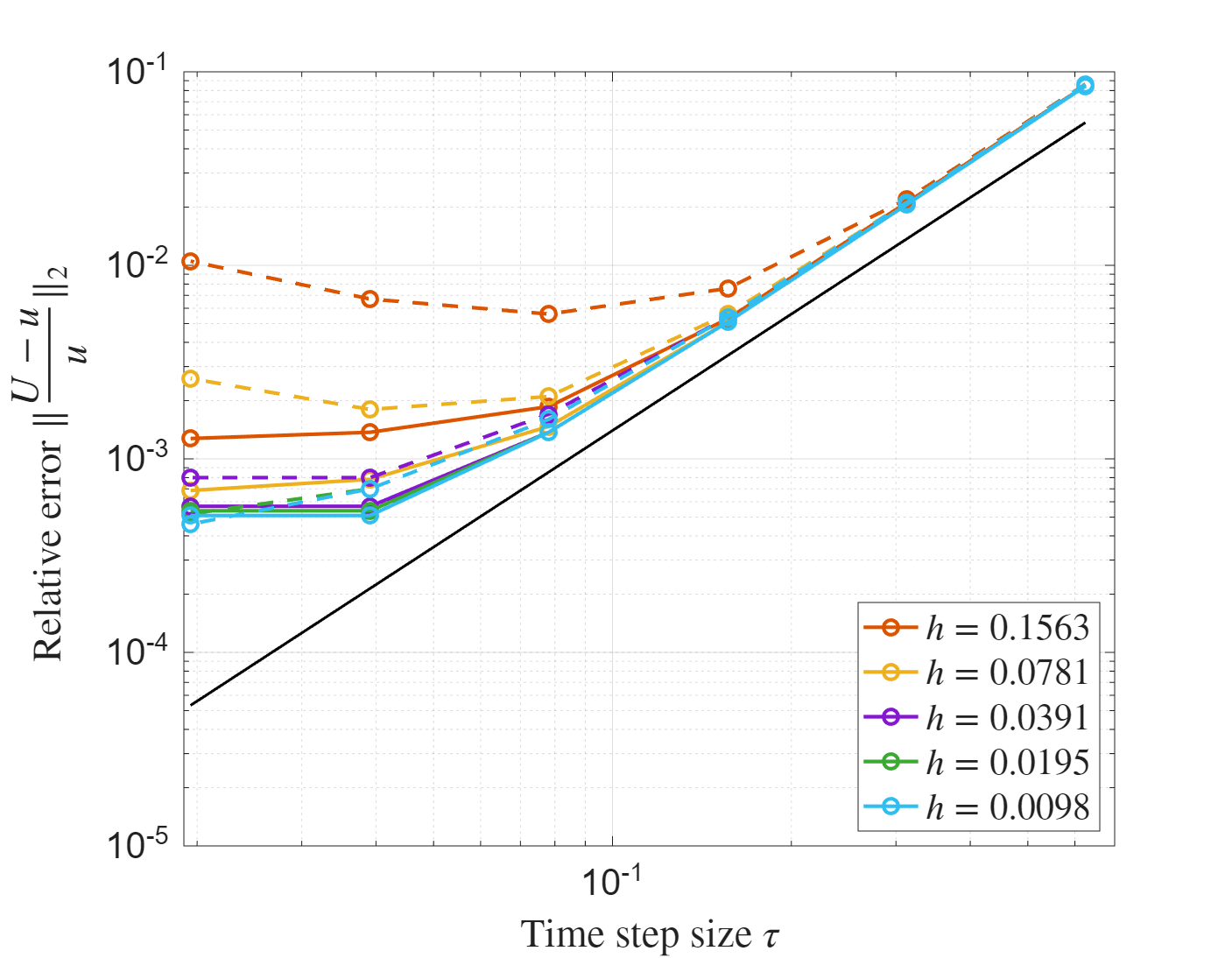}
    \end{minipage}
    \begin{minipage}{0.45\linewidth}
        \includegraphics[width=\linewidth]{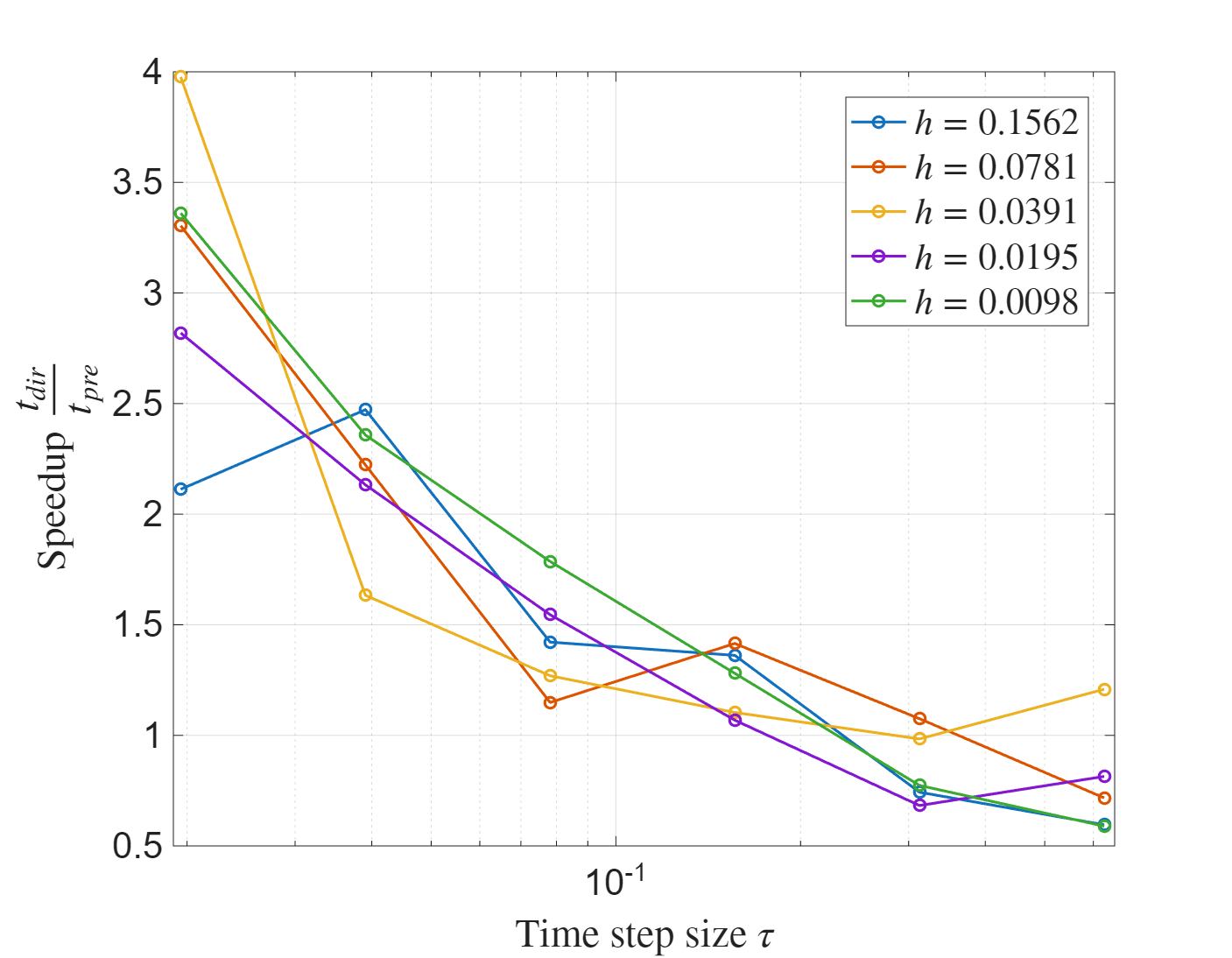}
    \end{minipage}
    \caption{Comparison at \(T=20\) for (\ref{heat}).
    Left: relative \(L^2\)-errors versus \(\tau\) for the preconditioned BVM (solid lines) and the standard BVM (dashed lines) with different $h=\frac{L}{128},\frac{L}{256},\frac{L}{512},\frac{L}{1024},\frac{L}{2048}$. The black line denotes the reference second-order convergence rate.
    Right: speedup ratio for the same spatial step sizes \(h\).}
    \label{converge_0}
\end{figure}

\begin{figure}
    \centering
    \includegraphics[width=0.5\linewidth]{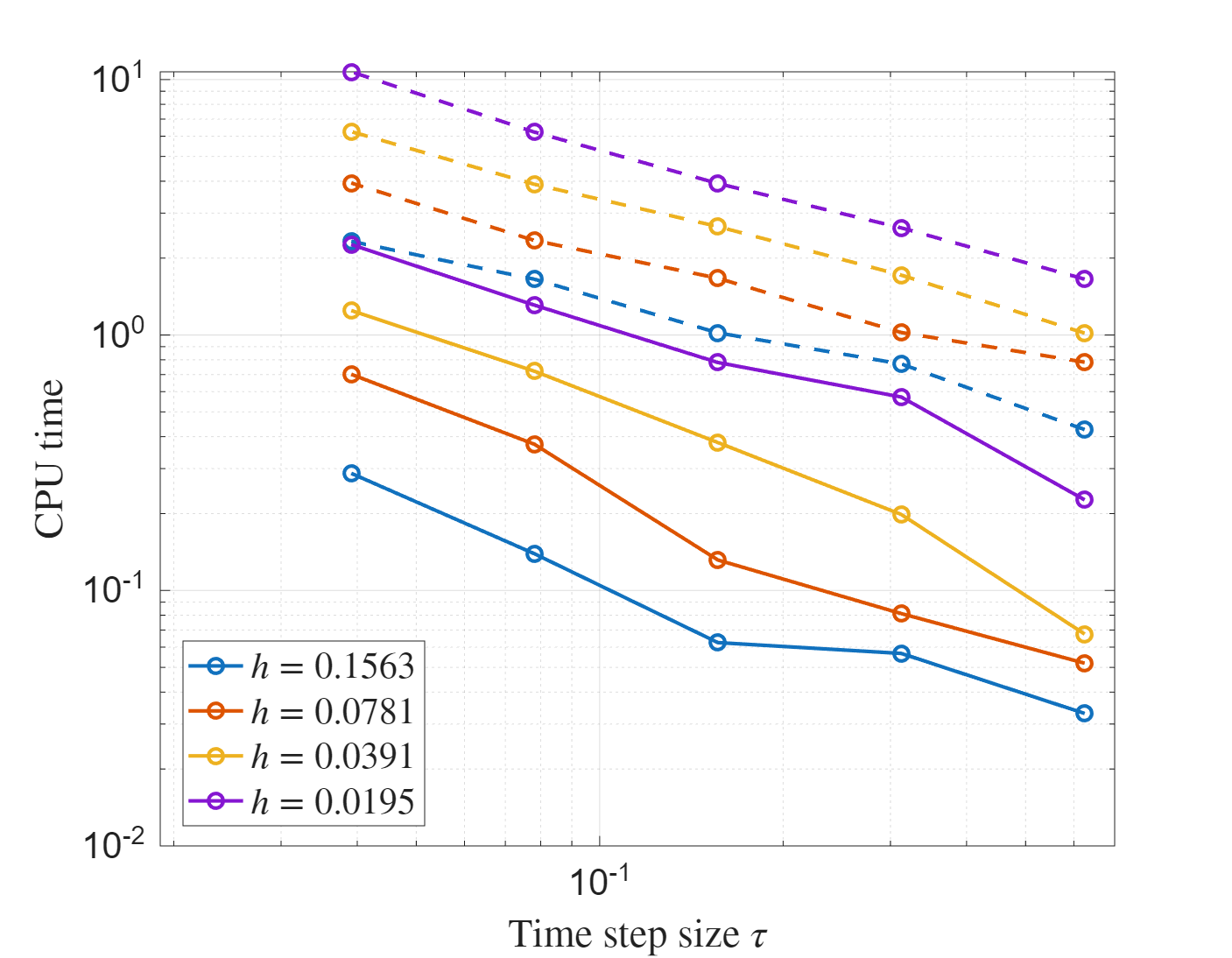}
    \caption{CPU time of GMRES with the preconditioner (solid lines) and without the preconditioner (dashed lines) for different $h=\frac{L}{128},\frac{L}{256},\frac{L}{512},\frac{L}{1024}$.}
    \label{cpu_compare}
\end{figure}

\subsection{The half-diffusion reaction equation ($\mathcal{L}=\delta\mathcal{I}$)}
Next, we examine a mass-transfer model that incorporates both diffusion and a linear reaction term. It describes unsteady transport in a quiescent medium, where the substance spreads while simultaneously undergoing a first-order reaction at rate $\delta$; negative values of $\delta$ correspond to loss or absorption. This framework is relevant, for instance, to solute transport with thermal or radioactive decay, and also to one-dimensional heat conduction with internal heat loss proportional to the temperature. In its fractional version, the model accounts for anomalously slow transport, often associated with trapping effects or viscoelasticity, while preserving the same linear reaction mechanism. The governing equation we consider is: 
\begin{equation}
   \left\{ \begin{aligned}
           & \partial_t u= -\epsilon(-\Delta)^\frac{1}{2}u + \delta u + f(x,t),~~(x,t)\in\RR\times(0,T],\\
        &u(x,0) = u_0(x).
    \end{aligned}\right.\label{diffusion_u}
\end{equation}

In this section, we choose $\epsilon = 0.1, \delta =-0.02$ and test the homogeneous case with the initial condition $u_0(x) = \frac{1}{(1+x^2)^2}$, the numerical diffusion process is presented in Figure \ref{u_sol}.  By choosing $f(x,t) = -\epsilon\cos(t)  \frac{  x^4 + 6x^2 - 3 }{2 \left( x^2 + 1 \right)^3} -\frac{\sin (t) +\delta \cos (t)}{(1+x^2)^2}$ with the same initial condition and same parameters, the exact solution is given by $u(x,t) = \frac{\cos(t)}{(1+x^2)^2}$. In Figure \ref{converge_u}, we compare the convergence behaviour of the GMM using the standard direct solver and the preconditioned GMRES solver. As in the case \(\mathcal{L}=0\), the preconditioned BVM shows better numerical performance than the standard BVM. In most tests, the preconditioned GMRES converges in about 17 iterations at a relative tolerance of $10^{-5}$.

\begin{figure}[htbp]
    \centering
    \begin{minipage}{0.45\linewidth}
        \includegraphics[width=\linewidth]{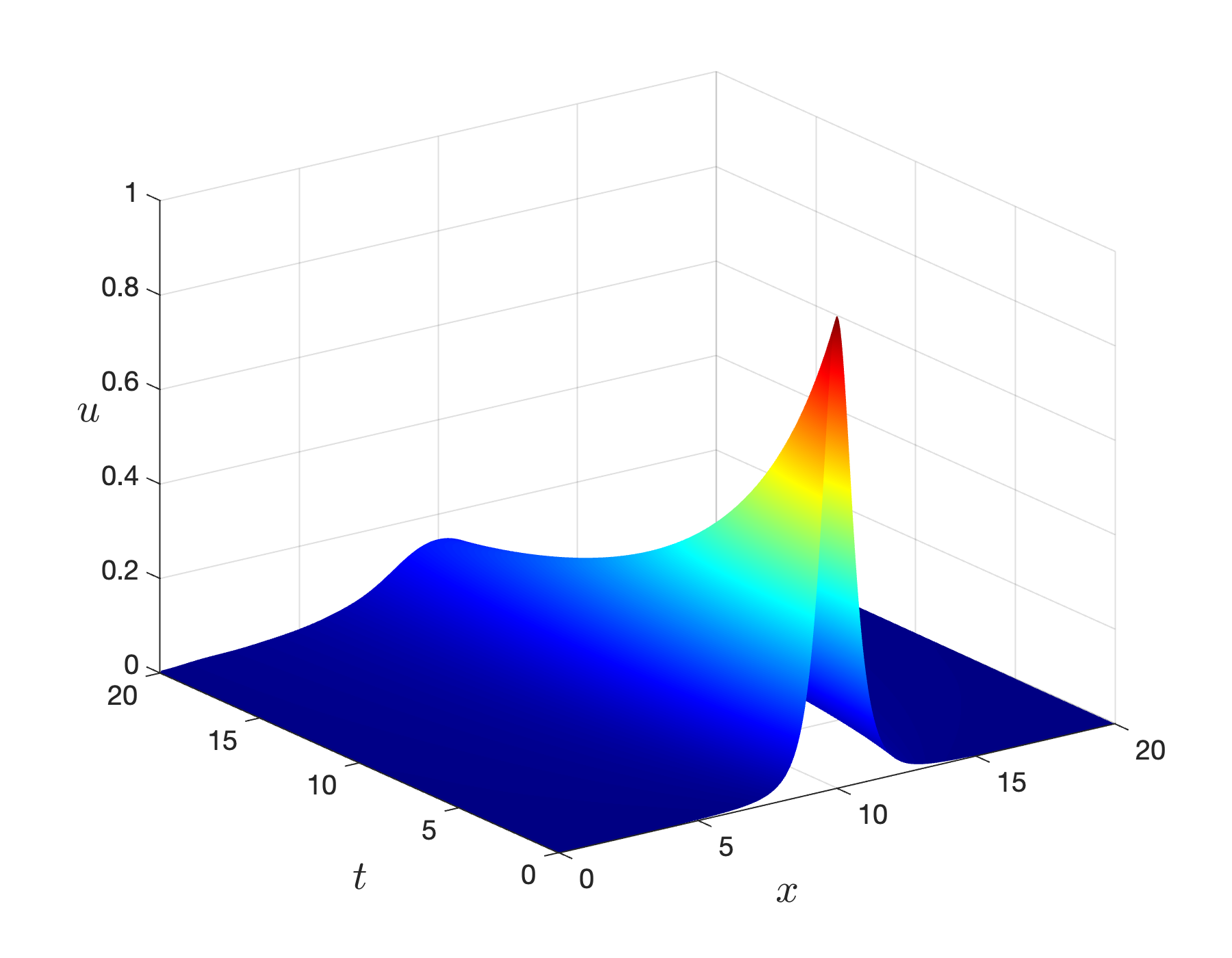}
    \end{minipage}
    \begin{minipage}{0.45\linewidth}
        \includegraphics[width=\linewidth]{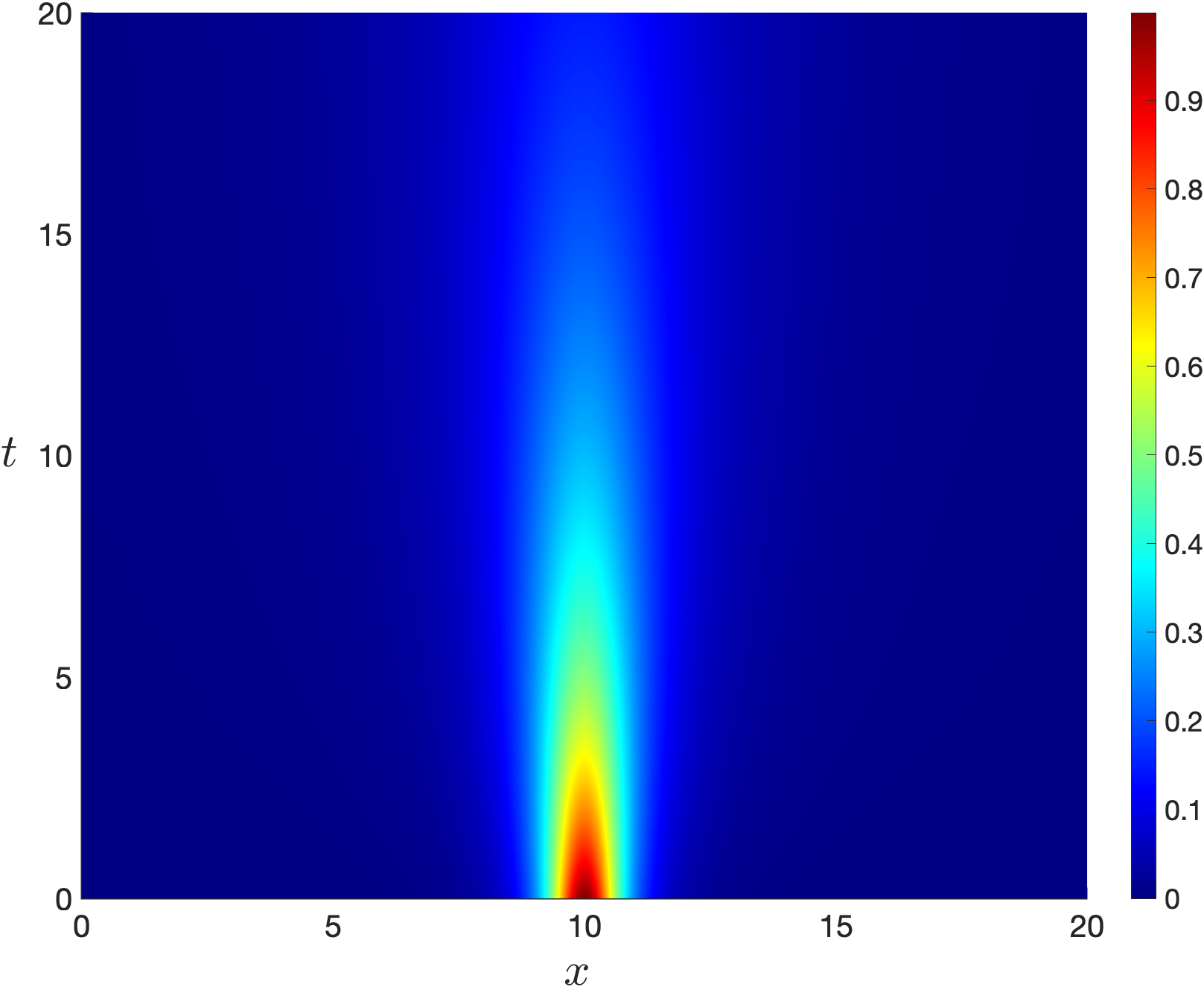}
    \end{minipage}
    \caption{The numerical solution of (\ref{diffusion_u}) with $\epsilon=0.1,\delta=0.02$, initialized with $u_0(x) = \frac{1}{(1+x^2)^2}$ and homogeneous Dirichlet boundary condition, with $\tau=0.0391$ and $h=0.0195$ at time $T=20$.}
    \label{u_sol}
\end{figure}

%%%%%%%%%%%%%%%
\begin{figure}[htb]
    \centering    
    \begin{minipage}{0.45\linewidth}
        \includegraphics[width=\linewidth]{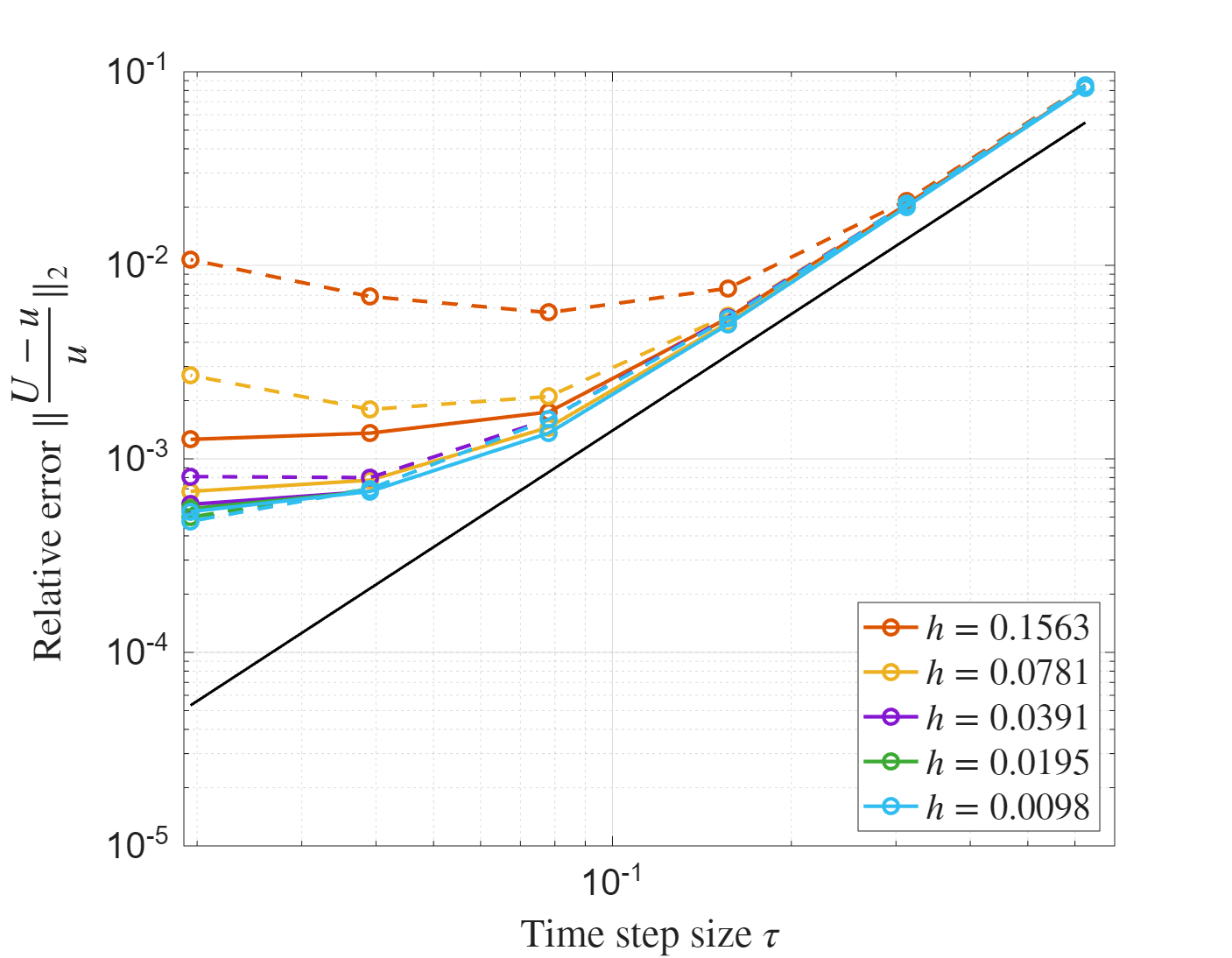}
    \end{minipage}
    \begin{minipage}{0.45\linewidth}
        \includegraphics[width=\linewidth]{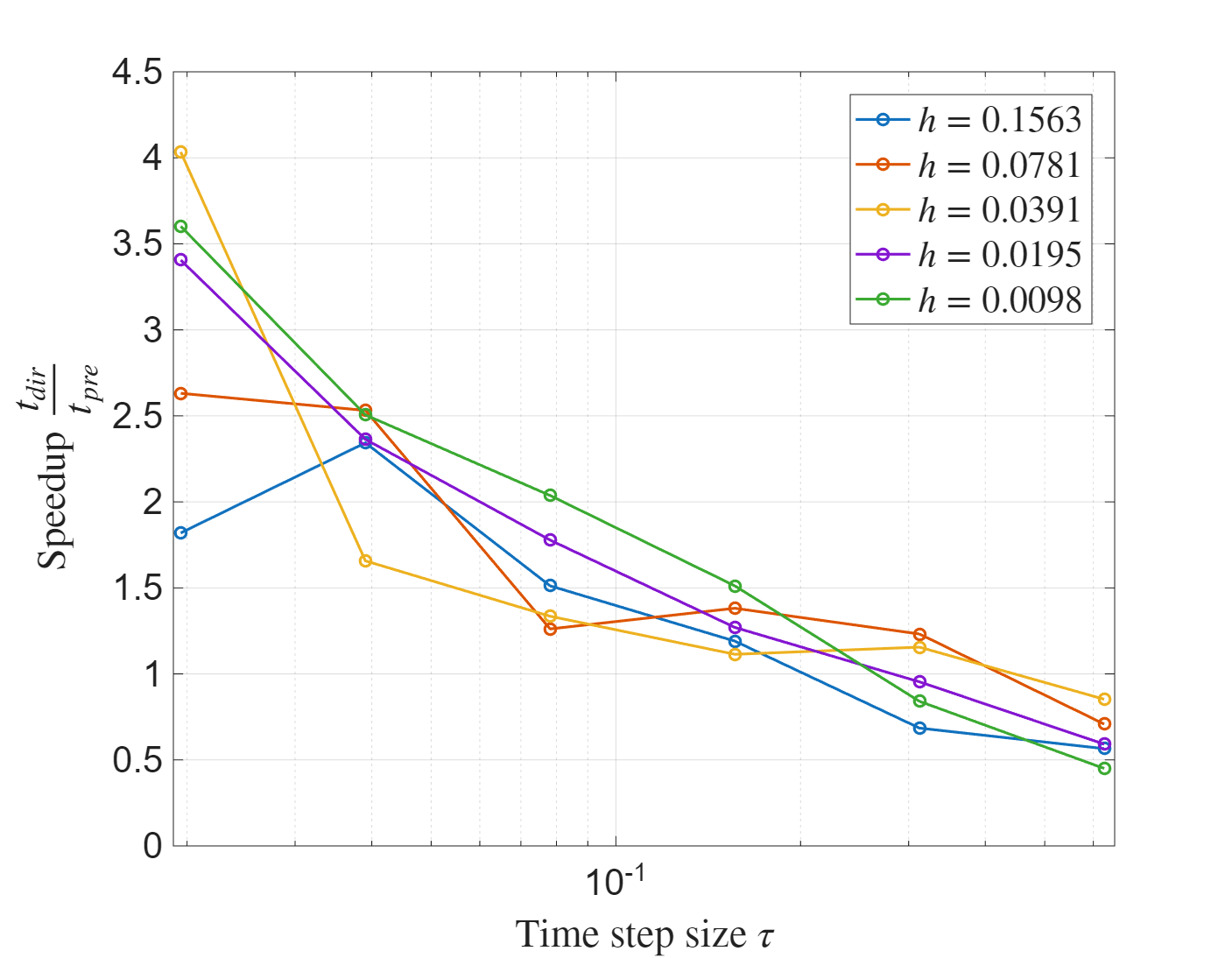}
    \end{minipage}
       \caption{Comparison at \(T=20\) for (\ref{diffusion_u}).
Left: relative \(L^2\)-errors versus \(\tau\) for the preconditioned BVM (solid lines) and the standard BVM (dashed lines) with different $h=\frac{L}{128},\frac{L}{256},\frac{L}{512},\frac{L}{1024},\frac{L}{2048}$. 
Right: speedup ratio for the same spatial step sizes \(h\).}
    \label{converge_u}
\end{figure}
\subsection{An advection dominated equation with half-diffusion ($\mathcal{L} = \delta \partial_x $)}
Finally, we turn to a half-diffusion model with advection, which describes one-dimensional, nonstationary mass transfer in which a systematic bias or flow is superimposed on jump-driven diffusion. Compared with the standard diffusion  $\Delta$, the operator $-(-\Delta)^{\frac12}$ provides much weaker diffusion. The advection term retains its classical first-order time dependence, while the absence of a source term ($f=0$) indicates that there is no absorption or release of substance during transport. The equation is given by: 
\begin{equation}
\left\{\begin{aligned}
    &\partial_t u= -\epsilon(-\Delta)^\frac{1}{2}u + \delta \partial_x u+ f(x,t),~~(x,t)\in \mathbb{T}_{2L}\times(0,T],\\
    &u(x,0) = u_0(x).
\end{aligned}\right.\label{diffusion_ux}
\end{equation}

To show the robustness of the preconditioned BVM, we take $\delta=0.2$ and $\epsilon=0.01$. In this setting, convection dominates the transport dynamics, while the half-diffusion term plays only a secondary role and provides weaker smoothing than the classical Laplacian $\Delta$. In advection-dominated regimes, standard central schemes may produce nonphysical oscillations near sharp fronts or discontinuities \cite{peclet}. In our computations, such oscillations appeared near the boundary as the final time increased, but they disappeared quickly as the spatial grid was refined. As in the previous two simulations, we use the same initial condition $u_0(x)$. Figure~\ref{sol_ux} shows the computed evolution of \eqref{diffusion_ux} without the source term $f(x,t)$; the numerical solution remains smooth and free of visible oscillations. In contrast to the two previous models, whose eigenvalues lie on the real axis, the eigenvalues of the matrix $\mathbf{D}$ associated with \eqref{diffusion_ux} spread along an $\infty$-shaped pattern, which is consistent with the analysis in Section~4.1, see Figure~\ref{eig_ux}. This spectral distribution further helps explain why many LMMs do not provide unconditional stability for this problem.
 
 Setting $f(x,t)=-\frac{\sin t}{(1+x^2)^2}-
\epsilon\cos t \frac{6x^2+x^4-3}{2(1+x^2)^3}$, $u(x,t) = \frac{\cos(t)}{(1+x^2)^2}$ is an exact solution. We provide the numerical convergence of the standard and the preconditioned BVM in the Figure \ref{converge_ux}, the performance of preconditioned and standard BVM are compared as well. Since the coefficient matrix 
$\mathbf{D}$ is highly ill-conditioned (for example, the condition number of $\mathbf{D}$ is $1.25\times10^{18}$ approximately when $(h,\tau)=(L/127,T/62)$), both the standard BVM and the preconditioned BVM would suffer from numerical deterioration. Under such circumstances, a smaller residual norm obtained by GMRES does not necessarily imply a smaller forward error; it merely reflects a better residual reduction for the discretized system. Overall, the preconditioned BVM tends to produce slightly more accurate solutions, but this improvement is accompanied by a higher computational cost, since the corresponding GMRES iteration does not converge rapidly in a small number of steps. Even so, for large-scale systems, the preconditioned BVM still exhibits better computational efficiency than the direct method, without any significant loss of accuracy.
 
 % Figure \ref{sol_ux} displays the computed evolution of (\ref{diffusion_ux}) without the source $f(x,t)$, the solution is smooth and free of oscillations. In contrast to the two earlier models, whose eigenvalues lie only on the real axis, the eigenvalues of the matrix $\mathbf{D}$ derived in (\ref{diffusion_ux}) spread out in a more complicated pattern, as shown in Figure \ref{eig_ux}, this helps explain why many LMMs fail.

\begin{figure}[htb] % previous used
    \centering
    \includegraphics[width=0.45\linewidth]{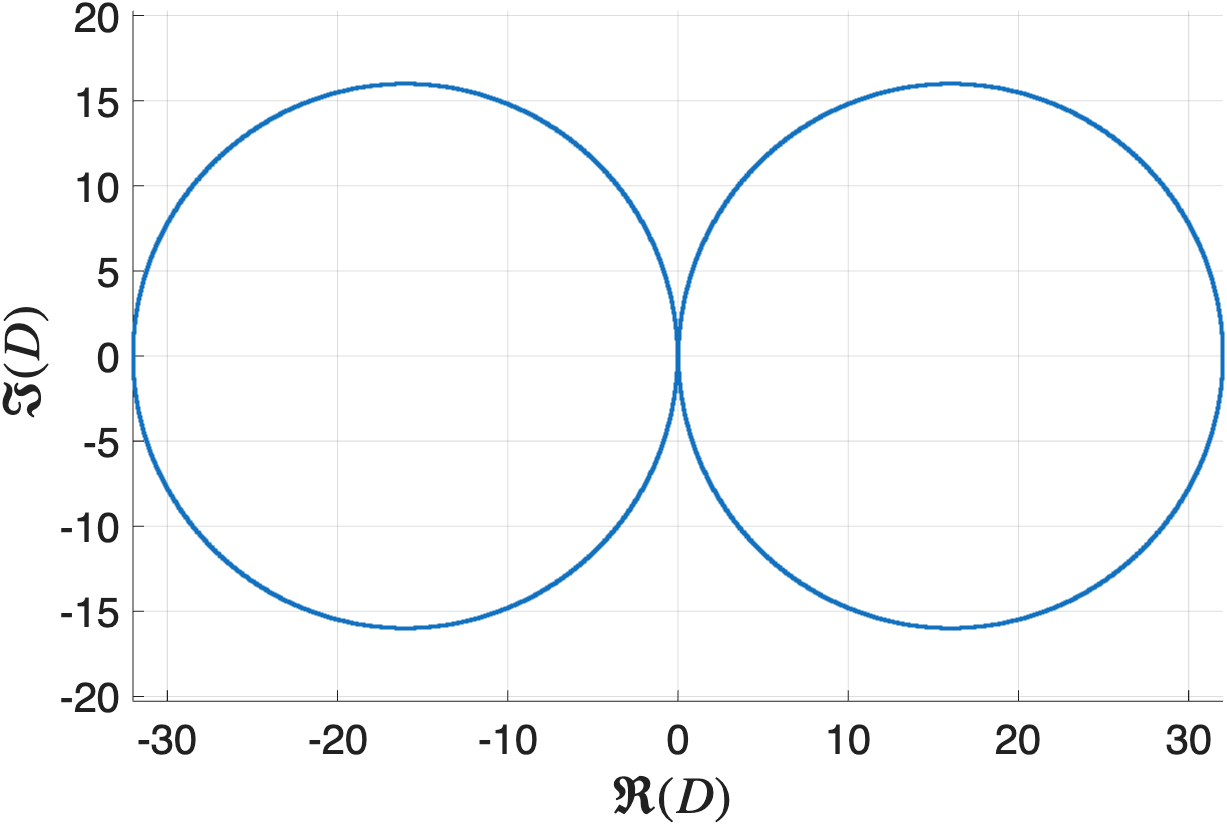}
    \caption{Distribution of the eigenvalues of the matrix $\mathbf{D}$, which is derived from the advection model (\ref{diffusion_ux}) with $\delta = 0.2$ and $\epsilon=0.01$.}
    \label{eig_ux}
\end{figure}

\begin{figure}[htb]
    \centering
    \begin{minipage}{0.45\linewidth}
        \includegraphics[width=\linewidth]{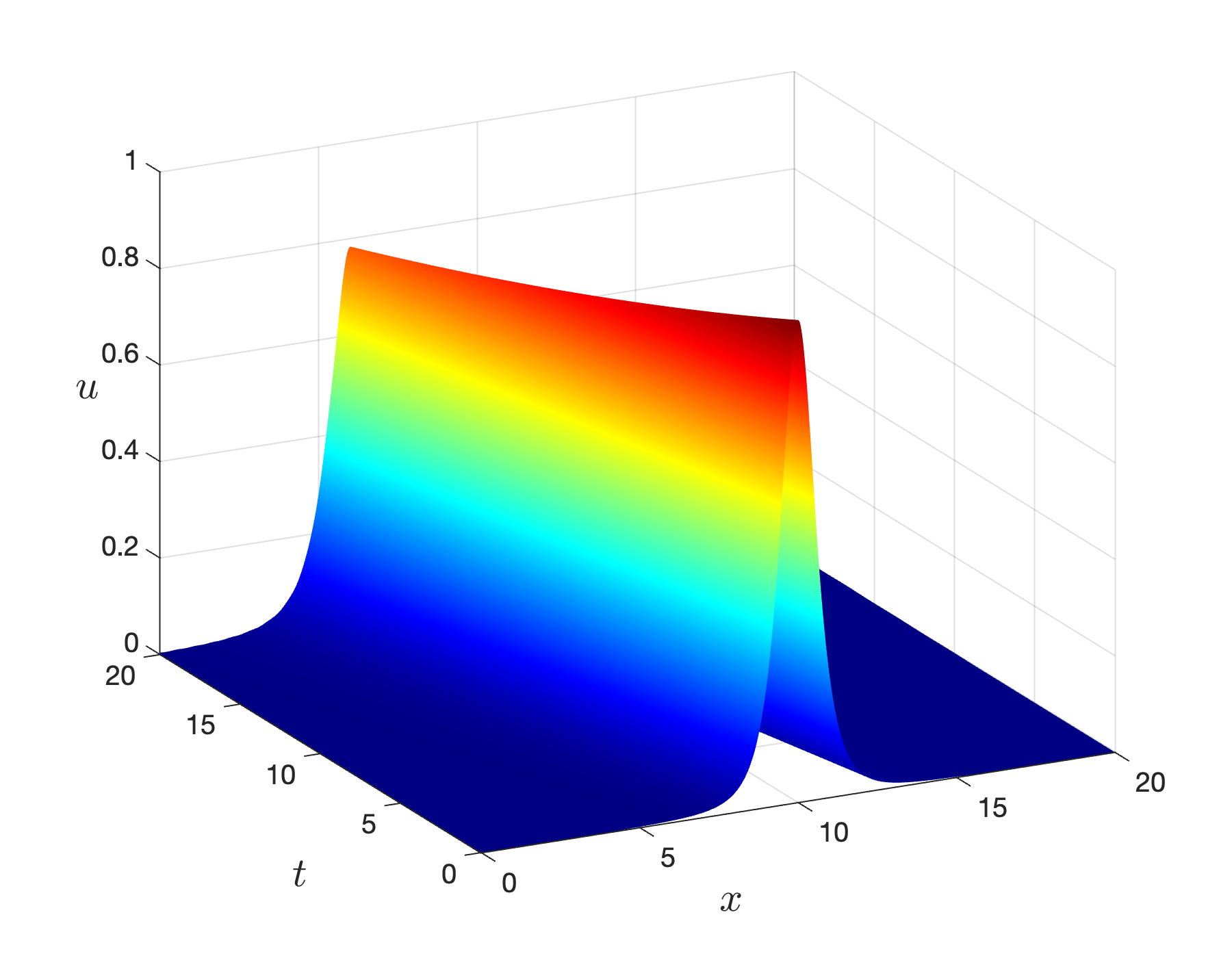}
    \end{minipage}
    \begin{minipage}{0.45\linewidth}
        \includegraphics[width=\linewidth]{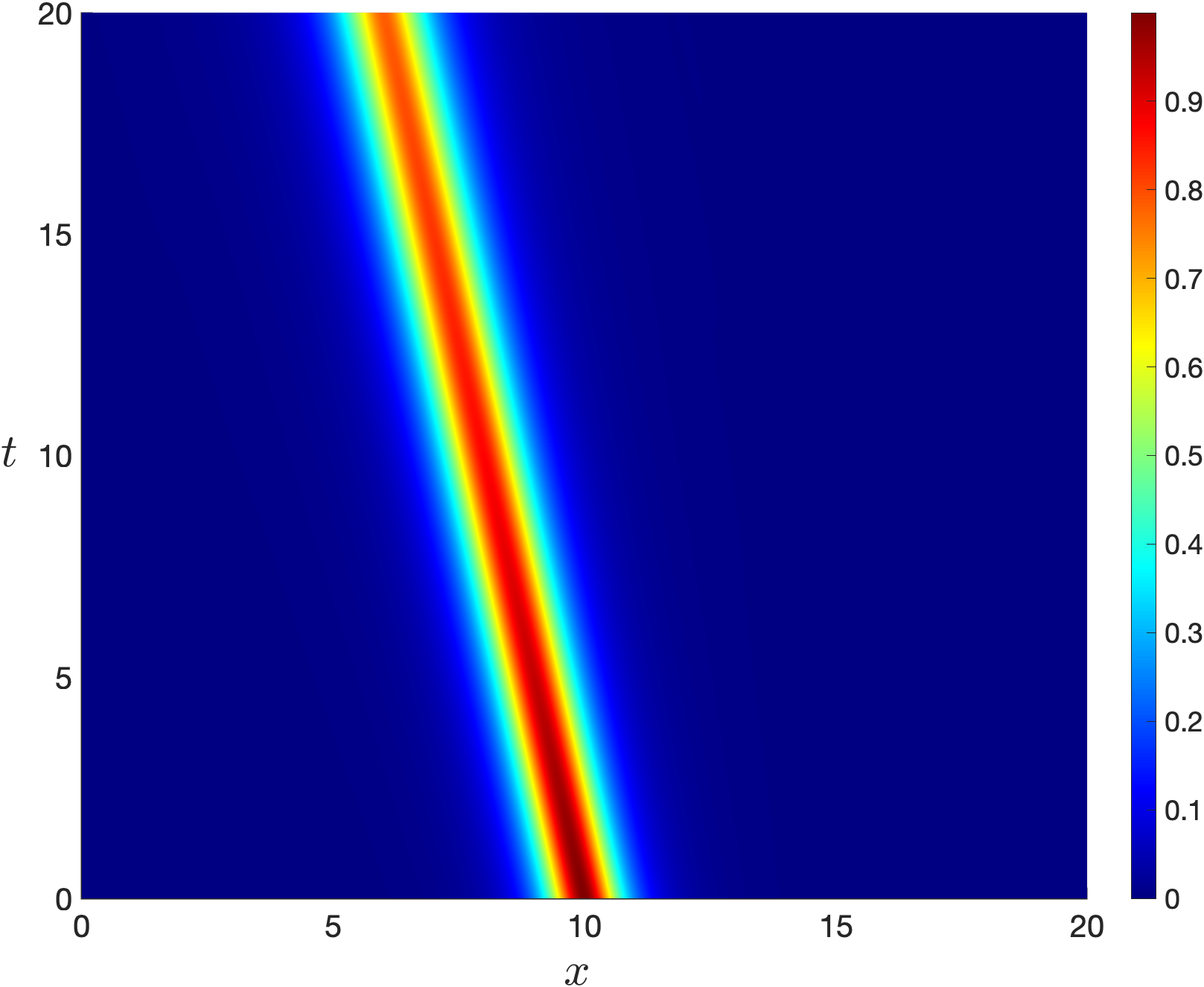}
    \end{minipage}
    \caption{The numerical solution of (\ref{diffusion_ux}) at time $T = 20$, computed with coefficients $\delta = 0.2,~\epsilon = 0.01$, using the preconditioned BVM with $\tau=0.0391$ and $h=0.0098$.}
    \label{sol_ux}
\end{figure}

\begin{figure}[htb]
    \centering    \begin{minipage}{0.45\linewidth}
        \includegraphics[width=\linewidth]{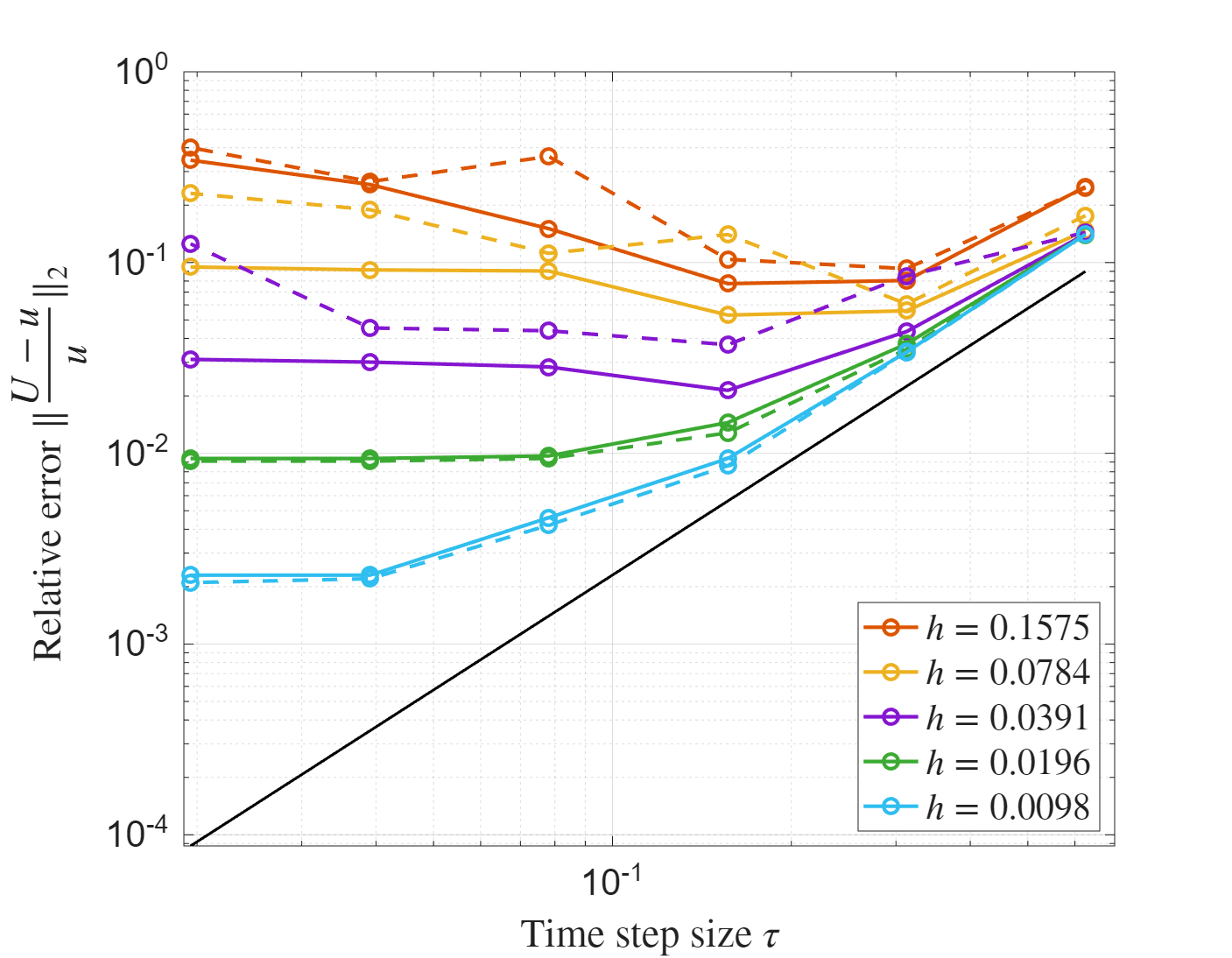}
    \end{minipage}
    \begin{minipage}{0.45\linewidth}
        \includegraphics[width=\linewidth]{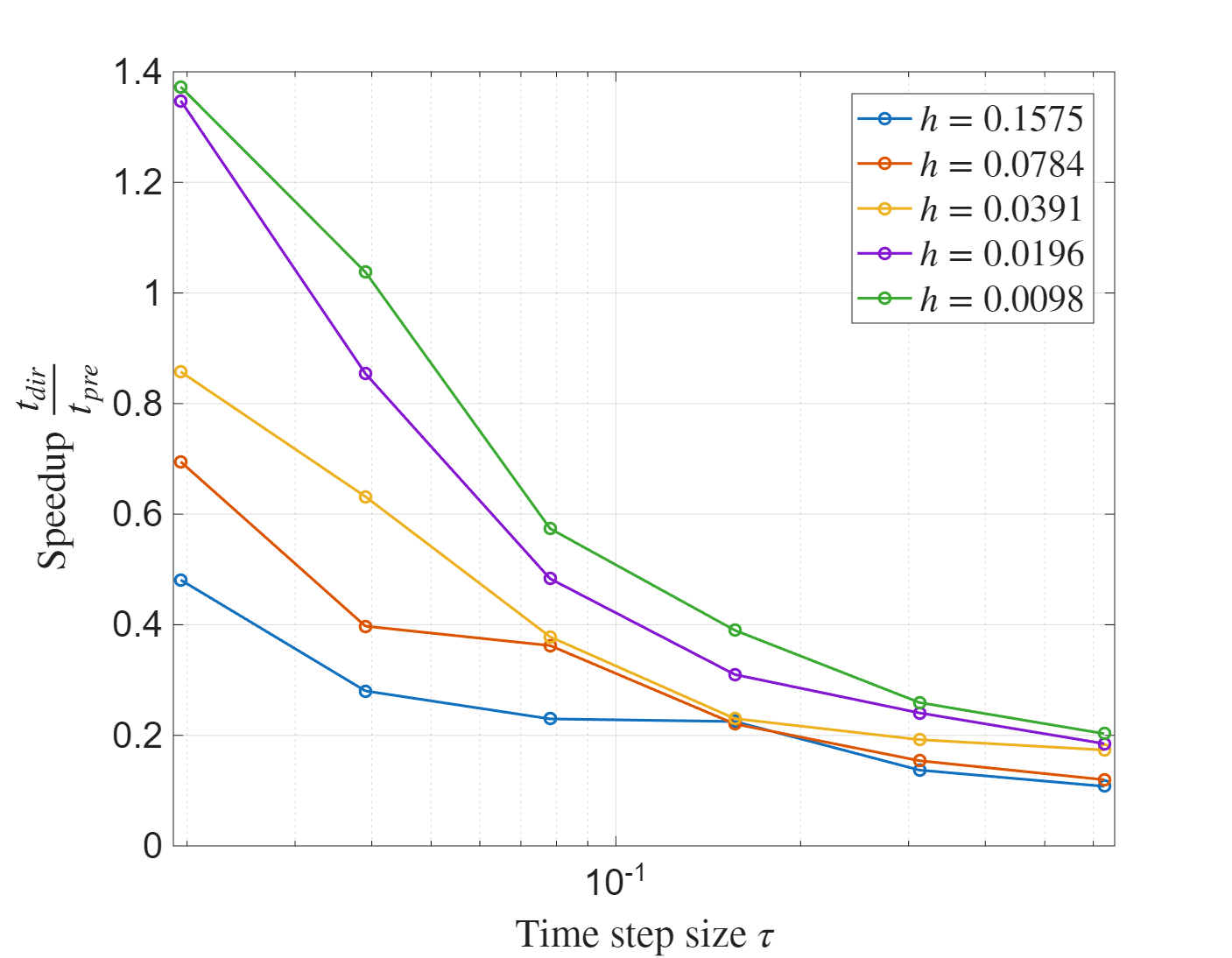}
    \end{minipage}
       \caption{Comparison at \(T=20\) for (\ref{diffusion_u}).
Left: relative \(L^2\)-errors versus \(\tau\) for the preconditioned BVM (solid lines) and the standard BVM (dashed lines) with different spatial steps $h=\frac{L}{127},\frac{L}{255},\frac{L}{511},\frac{L}{1023},\frac{L}{2047}$.
Right: speedup ratio for the same spatial step sizes \(h\).}
    \label{converge_ux}
\end{figure}

Using the same coefficients $\delta$ and $\epsilon$, we also present a numerical example in which neither the initial condition nor the source term admits a closed-form Hilbert transform. This example shows that the preconditioned BVM remains accurate and effective even when the Hilbert transforms are not available analytically. In this case, $\mathcal{H}(u_0)(x)$ and $\mathcal{H}(f)(x,t)$ are approximated by formula \eqref{compute_H}, and the corresponding numerical results are displayed in Figure~\ref{sol_ux2}. $u_0(x)=\frac{\exp(-(x-2)^4)}{1+(x-2)^2}$ and $f(x,t) = -\cos(t)\lb\frac{\exp(-(x-2)^4}{1+(x-2)^2}+\frac{\exp(-(x+2)^4}{1+(x+2)^2}\rb$ are adopted. 

\begin{figure}[htb]
    \centering
    \begin{minipage}{0.45\linewidth}
        \includegraphics[width=\linewidth]{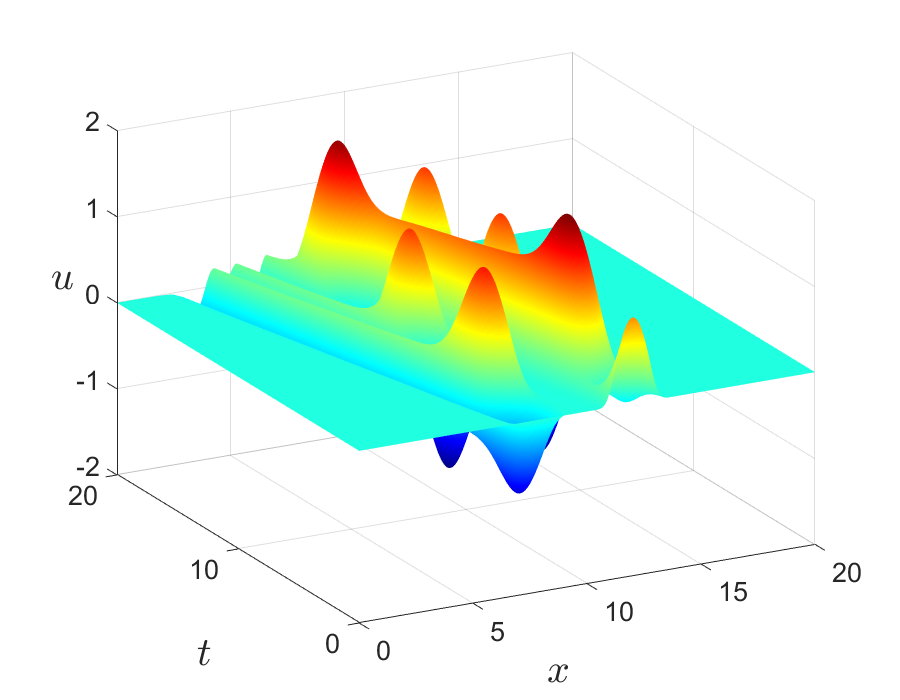}
    \end{minipage}
    \begin{minipage}{0.45\linewidth}
        \includegraphics[width=\linewidth]{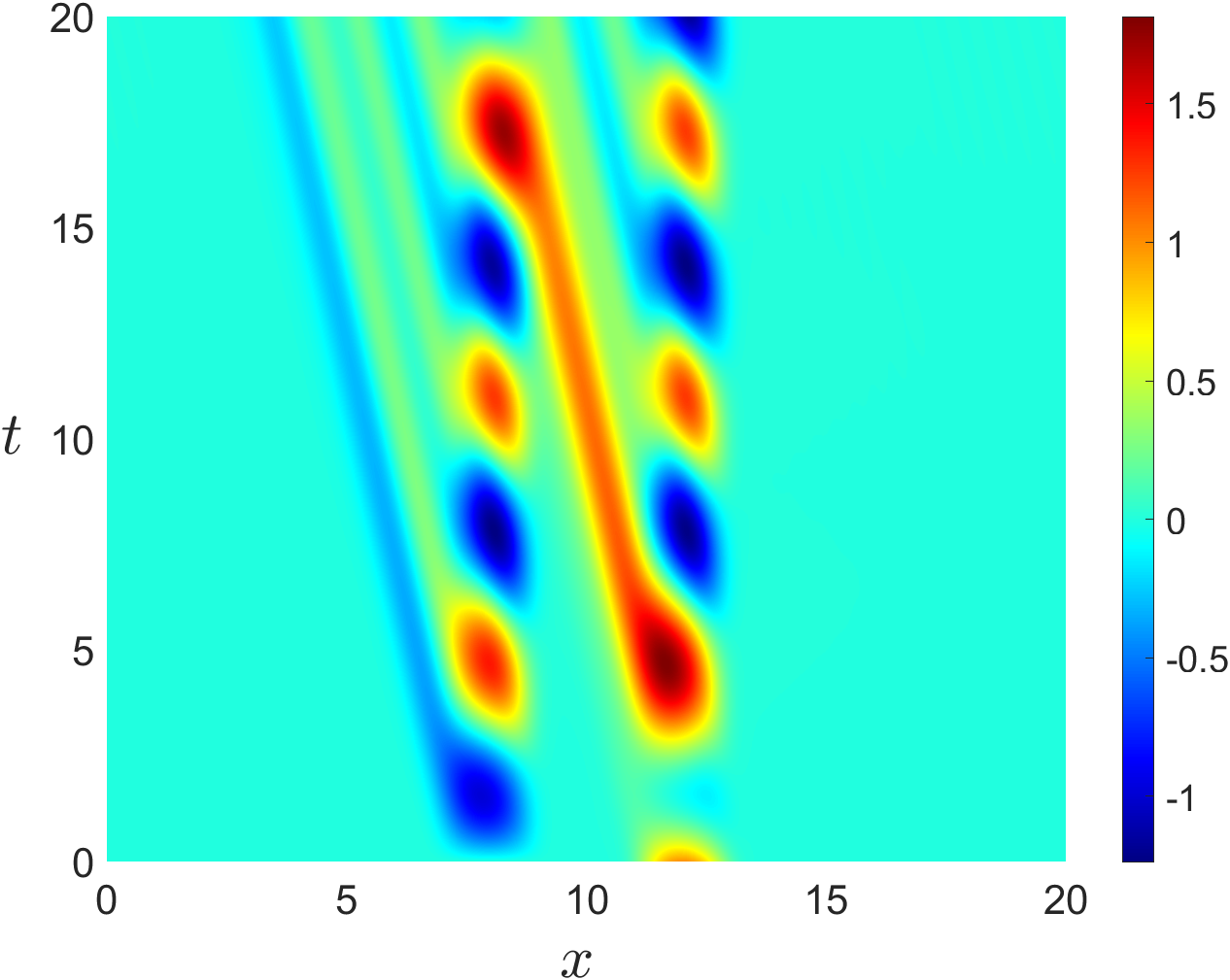}
    \end{minipage}
    \caption{Numerical solution of \eqref{diffusion_ux} at \(T=20\) with 
\(u_0(x)=\frac{e^{-(x-2)^4}}{1+(x-2)^2}\) and 
\(f(x,t)=-\cos t\bigl[\tfrac{e^{-(x-2)^4}}{1+(x-2)^2}+\tfrac{e^{-(x+2)^4}}{1+(x+2)^2}\bigr]\). 
Computed via preconditioned BVM with \(\delta=0.2\), \(\epsilon=0.01\), 
\(\tau=0.0391\), \(h=0.01953\), and relative \(L^2\)–error 0.0141.}
    \label{sol_ux2}
\end{figure}

\paragraph{Comparison to the standard transport equation $\partial_t u= \partial_x u$  ($\epsilon = 0$).}Here we should note that (\ref{diffusion_ux}) cannot degenerated as the normal transport equation $\partial_t u= \delta \partial_x u$ even if we set $\epsilon=0$. In fact, we have 
\begin{equation}
    \partial_t^2 u = -\delta^2 \Delta u - 2\delta \partial_t(\partial_x u) -\delta \partial_x f +\partial_t f.
\end{equation}
Therefore, we compare the numerical solution of (\ref{diffusion_ux}) with the transport equation for $\epsilon=0$ in Figure \ref{comapre_transport}. We take $\delta=1$ and run simulation for $t\in[0,2]$ and numerically confirmed \eqref{ds_2eqs} is compatible with the original problem \eqref{original_half}.
\begin{figure}[htb]
    \centering
    \begin{minipage}{0.45\linewidth} 
        \includegraphics[width=\linewidth]{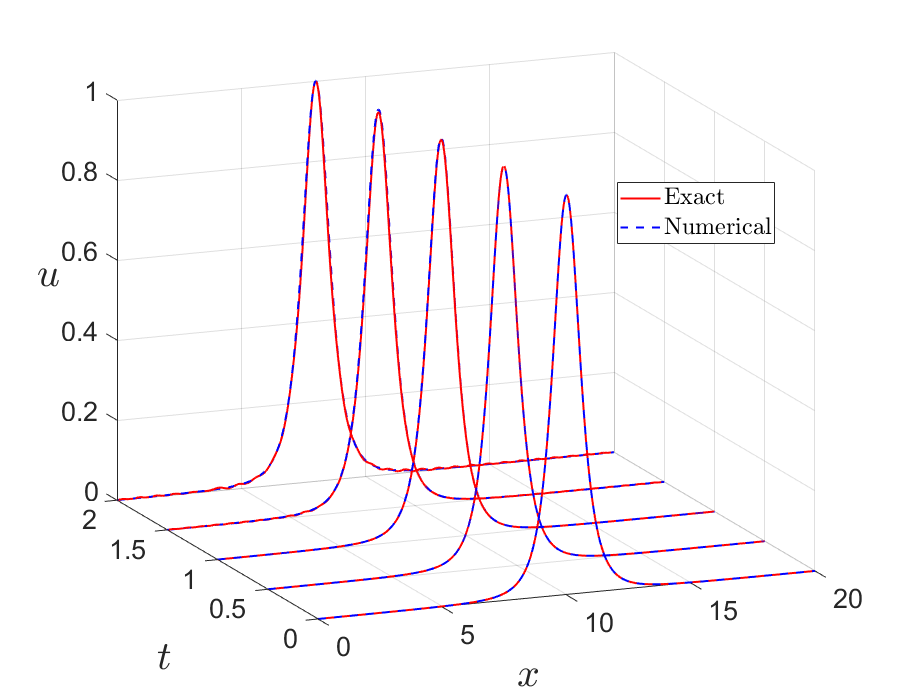}
    \end{minipage}
    \caption{Numerical comparison between zero-diffusion of (\ref{diffusion_ux}) and the standard transport equation $\partial_t u= \partial_x u$, using $h = \tau = 0.0391$ and initial condition $u_0(x) = \frac{1}{(1 + x^2)^2}$. }
    \label{comapre_transport}
\end{figure}

\section{A nonlocal Schrödinger equation}
\label{sec6}
As a special case we mentioned in section 3.3, $\epsilon$ is pure imaginary will lead a Schrödinger equation with the half-Laplacian. Suppose $\gamma = -\mathrm{i}\epsilon$ for $\epsilon\in\RR^+$, we consider 
\begin{equation}
    \left\{\begin{aligned}
   & \mathrm{i}~\partial_t u= \gamma(-\Delta)^\frac{1}{2}u + Vu,\qquad (x,t) \in \RR\times \RR^+,~ \gamma \in \RR,\\
   & u(x,0) = u_0(x),
\end{aligned}\right.\label{schrodinger}
\end{equation}
where $V$ is the potential term independent with $x$ and $t$. Here, we consider the case that $u\in H^1(\mathbb{T}_{2L})$ and $u\in H^1(\RR)$.
Physically, the propagator is unitary on $L^{2}(\mathbb{R})$ and the mass is conserved
\[
\|u(t)\|_{L^{2}(\mathbb{R})}=\|u_{0}\|_{L^{2}(\mathbb{R})}
\]
for $t\ge 0$.
The torus $\mathbb{T}_{2L}$ models a closed system on a ring of length $2L$.
There is no boundary but finite-size effects such as wrap-around and interference may arise over long times
due to periodicity and the nonlocal character of $(-\Delta)^{\frac12}$.  In fact, if $u\in H^1(\RR)$ then $\partial_x u\in L^2(\RR)$, and since $\Hh$ is bounded on $L^2(\RR)$,
it follows that $(-\Delta)^{\frac12}u\in L^2(\RR)$ \cite{ten_laplacian}. Therefore the right-hand side of \eqref{schrodinger} is well-defined in $L^{2}(\mathbb{R})$
and the Cauchy problem can be posed in $H^{1}(\mathbb{R})$.
A similar analysis applies to $H^{1}(\mathbb{T}_{2L})$.
In Theorem~\ref{wpd_srd} we prove well-posedness for \eqref{schrodinger} by first applying the doubling procedure and obtaining an equivalent wave system.

\begin{theorem}\label{wpd_srd}
    Suppose \(v_0, w_0 \in H^1(\RR) \text{ (or }H^1(\mathbb{T}_{2L}))\). Let $\widetilde{u} = e^{\mathrm{i}Vt}u$, then the solution to the following wave equation
    \begin{equation}
    \left\{\begin{aligned}
  & \widetilde{u}''=\gamma^2\partial_x ^2\widetilde{u},~~~~ (x,t) \in \RR\times \RR^+\\
  &  \widetilde{u}(x,0) = u_0(x),\widetilde{u}'(x,0) = -\mathrm{i}\gamma \mathcal{H}(\partial_x u_0(x))
\end{aligned}\right.\label{D_schrodinger}
\end{equation}
    is also solves (\ref{schrodinger}), that is,
    \begin{equation}
        u(x,t) = \frac{e^{\mathrm{-i}Vt}}{2}[u_0(x+\gamma t)+u_0(x-\gamma t)-\mathrm{i}~ \mathcal{H}(u_0(x+\gamma t))+\mathrm{i}~ \mathcal{H}(u_0(x-\gamma t))]\label{sol_schrodinger}
    \end{equation}
\end{theorem}

\begin{proof}
As similar in Appendix B.2, by letting $u = e^{-\mathrm{i}Vt}\widetilde{u}$ to cancel the term $Vu$ in (\ref{schrodinger}). We then use the SD procedure of (\ref{schrodinger}) and arrive at the wave system \eqref{D_schrodinger}. From \cite{acp_book}, the solution to \eqref{sol_schrodinger} is uniquely existed, i.e., $\tilde{u}\in C\big([0,T];H^2\big)\cap C^1\big([0,T];H^1\big)\cap C^2\big([0,T];L^2\big).$ Specificlly, writing $\widetilde{u}=v+\mathrm{i} w$ and \( u_0 = v_0+ \mathrm{i} w_0\) we arrive at two real wave system with respect to $v,w$. By using the identity $\mathcal{H}(\partial_x u_0)=\partial_x \mathcal{H}(u_0)$ along with d'Alembert's formula, we obtain the solution to these two system
\begin{align}
     &   v(x,t) = \frac{1}{2}\left(v_0(x+\gamma t)+v_0(x-\gamma t)\right) +\frac{1}{2}\left(\mathcal{H}\left(w_0(x+\gamma t)\right)-\mathcal{H}\left(w_0(x-\gamma t)\right)\right),\tag{6.5a}\label{v}\\
      &  w(x,t) = \frac{1}{2}\left(w_0(x+\gamma t)+w_0(x-\gamma t)\right) -\frac{1}{2}\left(\mathcal{H}\left(v_0(x+\gamma t)\right)-\mathcal{H}\left(v_0(x-\gamma t)\right)\right),\tag{6.5b}\label{w}
\end{align}
which directly give the solution of (\ref{D_schrodinger}):
\[\widetilde{u}(x,t) = \frac{1}{2}[u_0(x+\gamma t)+u_0(x-\gamma t)-\mathrm{i} \mathcal{H}(u_0(x+\gamma t))+\mathrm{i} \mathcal{H}(u_0(x-\gamma t))].\]
Finally multiplying by the factor $e^{-\mathrm{i}Vt}$ we obtain (\ref{sol_schrodinger}). Then we claim that the solution of (\ref{D_schrodinger}) also solves (\ref{schrodinger}). Indeed, for the equivalent Cauchy problem

\begin{equation}\label{eq:free_half_sch_proof}
i \partial_t w=\gamma(-\Delta)^{\frac12}w,\qquad w(x,0)=u_0(x),
\end{equation}
We take the Fourier transform in $x$ with $\widehat f(\xi)=\int_{\mathbb R}e^{-ix\xi}f(x) dx$ gives
\[
i \partial_t \widehat w(\xi,t)=\gamma|\xi| \widehat w(\xi,t),
\]
hence
\[
\widehat w(\xi,t)=e^{-i\gamma|\xi|t}\widehat u_0(\xi)
=
e^{-i\gamma\xi t}\mathbf 1_{\{\xi>0\}}\widehat u_0(\xi)
+
e^{ i\gamma\xi t}\mathbf 1_{\{\xi<0\}}\widehat u_0(\xi).
\]
From $\widehat{\mathcal H f}(\xi)= -i \mathrm{sgn}(\xi) \widehat f(\xi)$ we define the frequency projections
\[
\mathcal{P}^+:=\frac12(I+i\mathcal H),\qquad \mathcal{P}^-:=\frac12(I-i\mathcal H).
\]
Then $\widehat{\mathcal{P}^+f}=\mathbf 1_{\{\xi>0\}}\hat f$ and $\widehat{\mathcal{P}^-f}=\mathbf 1_{\{\xi<0\}}\hat f$, so with
$u_0^\pm:=\mathcal{P}^\pm u_0$ we can rewrite
\[
\widehat w(\xi,t)=e^{-i\gamma\xi t}\widehat{u_0^+}(\xi)+e^{ i\gamma\xi t}\widehat{u_0^-}(\xi),
\]
By the Fourier shift property,
$\mathcal F^{-1}(e^{-i\xi a}\hat f)(x)=f(x-a)$, it follows that
\begin{equation}
w(x,t)=u_0^+(x-\gamma t)+u_0^-(x+\gamma t).\label{sdg_solution}
\end{equation}
Substituting $u_0^\pm$ into \eqref{sdg_solution} and multiplying by $e^{-\mathrm{i}Vt}$, we obtain \eqref{sol_schrodinger}.
\end{proof}

\begin{remark}
The formula \eqref{sol_schrodinger} shows that in both models the solution is a superposition of two profiles translated at speeds $\pm\gamma$. In practice, if $u_0$ is well localized and supported away from the periodic junction where the endpoints are identified, and if $L$ is sufficiently large, then the Hilbert transform of $u_0$ is nearly the same on any compact subset of $(0,2L)$ whether $u_0$ is regarded as an element of $H^1(\mathbb{T}_{2L})$ or of $H^1(\mathbb{R})$.
Consequently the two evolutions are nearly indistinguishable up to the wrap-around time scale, namely until the shifts $x\pm\gamma t$ reach the junction modulo $2L$.
For larger times the periodic model exhibits wrap-around and self-interference, whereas the whole-line model continues to propagate on $\mathbb{R}$ without such effects.

\end{remark}

\begin{corollary}
    \(v_0, w_0 \in H^2(\mathbb{T}_{2L}) \), then (\ref{sol_schrodinger}) is equivalent to
    \begin{equation}
        u(x,t) = \frac{2}{L}\sum^\infty_{n=0}C_n\sin\left(\frac{\pi n x}{L}\right)e^{\mathrm{i} (\gamma\pi n/L-V) t},
    \end{equation}
    where
    \(C_n=\int^L_0 u_0(\xi)\sin\left(\frac{\pi n \xi}{L}\right)d\xi.\)
\end{corollary}
We present two numerical examples of (\ref{schrodinger}) in Figure \ref{fig:schrodinger} with $\epsilon=0.1$ and initial condition
\[u_0(x)= \frac{2}{1+(x+8)^2} - \frac{5i}{1+(x-8)^2}.\]
The solution is computed over the domain $x\in(0,48)$ from $T=0$ to $T=20$, comparing a zero potential and a constant potential $V=1$.
\begin{figure}[htb]
    \centering
    \begin{minipage}{0.4\linewidth}
        \includegraphics[width=\linewidth]{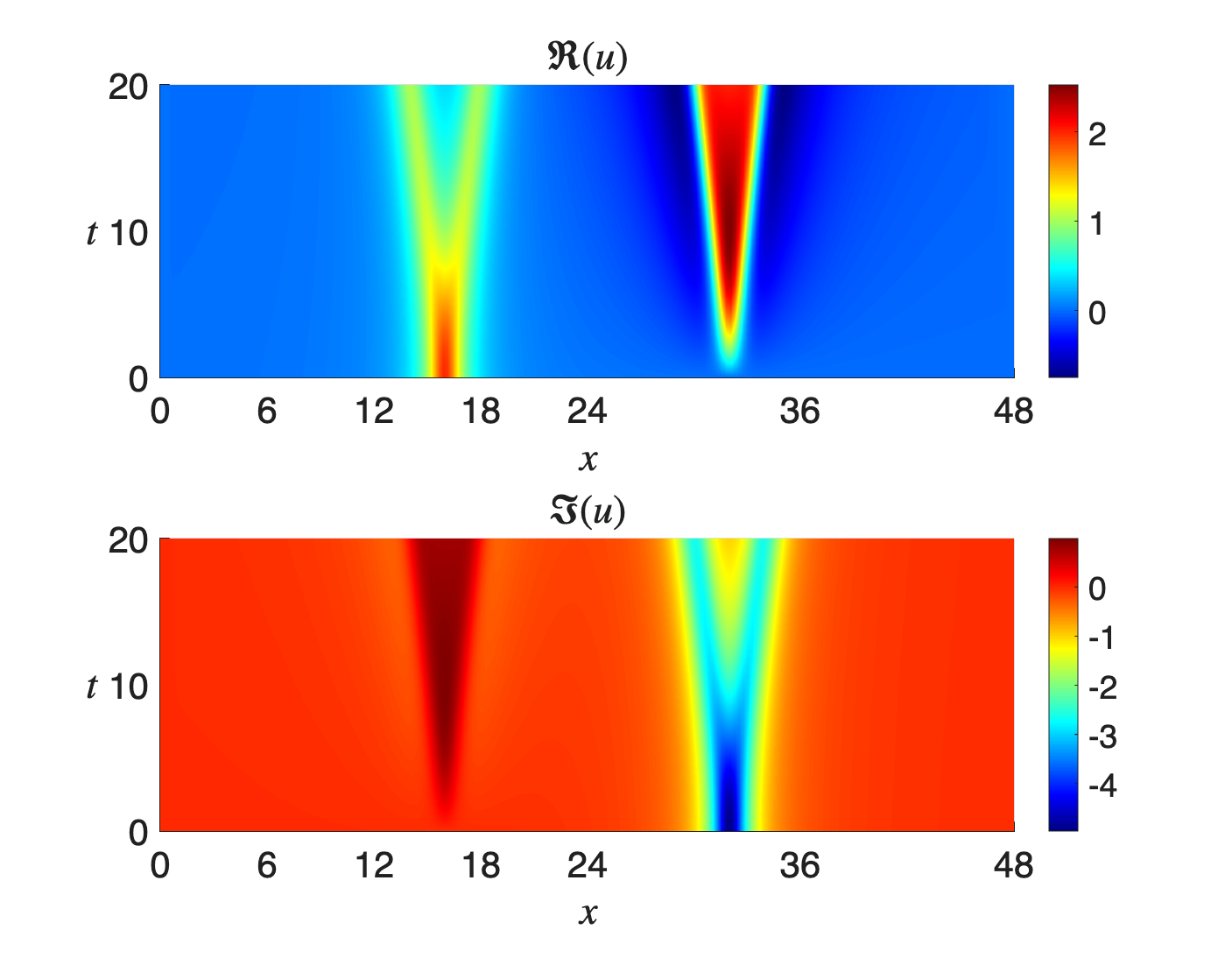}
    \end{minipage}
    \quad
    \begin{minipage}{0.4\linewidth}
        \includegraphics[width=\linewidth]{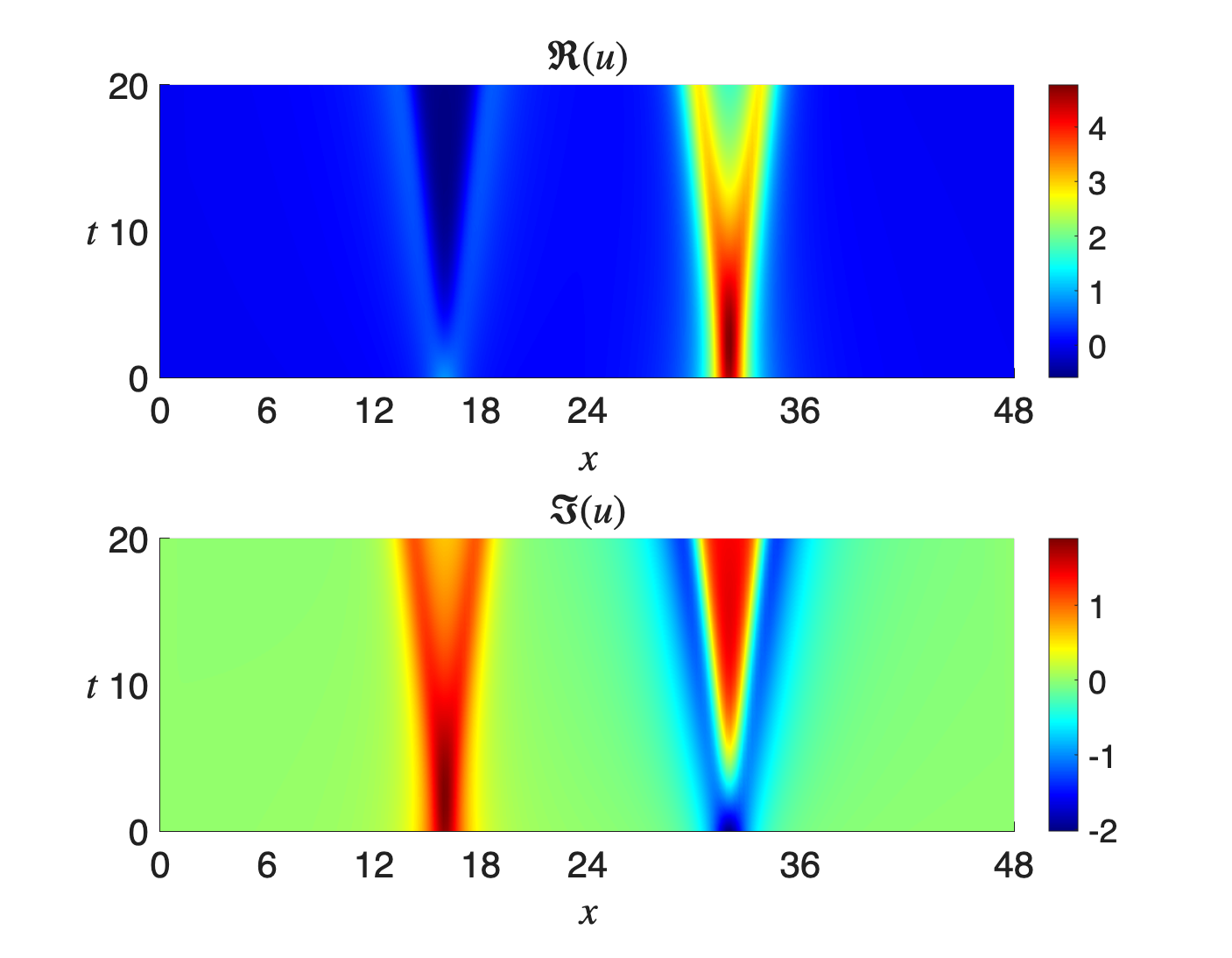}
    \end{minipage}
\caption{Solution for $\epsilon=0.1$ on $x\in(0,48)$ for $t\in[0,20]$. Left: No potential ($V=0$). Right: Constant potential ($V=1$).}
    \label{fig:schrodinger}
\end{figure}

\section{Conclusion}
\label{sec7}
In this paper, we introduced a Spectrum Doubling reformulation for a class of evolutionary PDEs involving the half-Laplacian. By exploiting the identity between the half-Laplacian and the Hilbert transform, the original first-order half-diffusion equation was recast as an equivalent doubled system. This reformulation reveals a  spectral splitting into stable and unstable branches. We showed that, under a suitable compatibility condition on the initial condition, the original half-diffusion dynamics is recovered precisely on a stable invariant subspace of the doubled system. In this sense, the SD framework provides not only an alternative representation of the problem, but also a clearer structural interpretation of the underlying dynamics. On the numerical side, the SD reformulation converts the evolution problem into an integer-order system, which makes it possible to apply time discretization techniques developed for classical PDEs. Since standard time integrators may suffer from stability difficulties for the resulting backward wave-type structure, we adopted Boundary Value Methods and established second-order temporal convergence together with eigenvalue-based stability properties. For the large Kronecker-structured linear systems arising from the full discretization, we further employed preconditioned iterative solvers, leading to an efficient implementation. The numerical experiments confirmed the theoretical analysis and demonstrated that the proposed approach provides accurate approximations for several representative half-diffusion models.

The present work suggests that reformulation-based approaches of this type may be useful for fractional evolutionary equations in the low-order regime, where nonlocality and singular behavior often make both the analysis and the numerical treatment considerably more delicate than in the classical or higher-order diffusive setting. A natural direction for future research is therefore to investigate whether the SD viewpoint can be extended beyond the half-Laplacian to more general operators of fractional order, as well as to higher-dimensional problems and more involved models such as fractional Schr\"odinger-type equations and fractional dispersion equations.
\subsection*{Funding}
This publication is supported by China Scholarship Council (Grant No. 202207720037).

\subsection*{Author Contributions Statement}
P.Yuan and P.Zegeling designed the study and developed the methodology. P.Yuan carried out the theoretical analysis and performed the numerical simulations. P.Yuan and XM.Gu contributed to the interpretation of results and validation. P.Yuan, P.Zegeling and XM.Gu wrote the main manuscript text. All authors discussed the results and reviewed the manuscript.

\begin{appendices}

\section{Some Functions and their Hilbert Transforms}\label{App_a}
This section lists several common functions together with their exact Hilbert transforms for quick reference.
\begin{table}[H]
    \centering
    \caption{Examples of functions and their Hilbert transforms}\label{table_hilbert_transform}%
    \begin{tabular}{@{}cc@{}}
        \toprule
        Function \( f(x),\ \Re(\alpha)>0 \) & Hilbert Transform \( \mathcal{H}(f(x)) \) \\ 
        \midrule
        \( \frac{1}{1+x^2} \) & \( \frac{x}{1+x^2} \) \\ 
        \( \frac{1}{1+x^4} \) & \( \frac{x(x^2+1)}{\sqrt{2}(x^4+1)} \) \\ 
        \( \frac{1}{(1+x^2)^2} \) & \( \frac{x(x^2+3)}{2(1+x^2)^2} \) \\ 
        \( e^{-\alpha x^2} \) & \( -\text{erf}(\mathrm{i}\sqrt{\alpha}x)e^{-\mathrm{i}\alpha x^2} \) \\ 
        \( \cos(\alpha x) \) & \( \sin(\alpha x) \) \\ 
        \( \sin(\alpha x) \) & \( -\cos(\alpha x) \) \\ 
        \( \frac{x}{x^2+\alpha^2} \) & \( -\frac{\alpha}{x^2+\alpha^2} \) \\
        \bottomrule
    \end{tabular}
\end{table}

\section{Analytical solutions for selected cases of \eqref{original_half}}
Appendix B provides explicit solutions for \eqref{original_half} on $\mathbb{R}$ for three specific forms of $\mathcal{L}$. By applying the spatial Fourier transform, we decouple the governing equation into a family of ODEs, yielding closed-form representations via Fourier multipliers. For the $2L$-periodic case, the same derivation holds by replacing the Fourier transform with Fourier series.

\subsection{The inhomogeneous half-diffusion equation: $\mathcal{L}=0$}

In Section 3, we already showed that the solution to the homogeneous equation corresponding to \eqref{original_half} with $\mathcal{L}=0$ and $f=0$ is
\[
u(x,t)=\frac1{2\pi}\int_{\mathbb R}\widehat u_0(\xi) e^{i\xi x}e^{-\epsilon |\xi|t} d\xi,
\qquad (x,t)\in\mathbb R\times[0,T],
\]
where
\[
\widehat u_0(\xi)=\int_{\mathbb R}u_0(y)e^{-i\xi y} dy.
\]
We now consider the inhomogeneous problem, namely $f\neq 0$. Taking the Fourier transform with respect to $x$, and applying the method of variation of constants, we obtain
\[
\widehat u(\xi,t)
=e^{-\epsilon |\xi|t}\widehat u_0(\xi)
+\int_0^t e^{-\epsilon |\xi|(t-s)}\widehat f(\xi,s) ds.
\]
Applying the inverse Fourier transform yields
\begin{equation}
\label{mild_half_solution}
u(x,t)
=\frac1{2\pi}\int_{\mathbb R}e^{i\xi x}e^{-\epsilon |\xi|t}\widehat u_0(\xi) d\xi
+\frac1{2\pi}\int_{\mathbb R}\int_0^t
e^{i\xi x}e^{-\epsilon |\xi|(t-s)}\widehat f(\xi,s) ds d\xi.
\end{equation}

As for the periodic case $u\in\mathbb{T}_{2L}$, it can be regarded as the discrete spectral counterpart of the whole-line problem, obtained by restricting the continuous spectrum to the admissible periodic frequencies. The same argument applies with the Fourier transform replaced by the Fourier series, so that the continuous frequency variable \(\xi\in\mathbb R\) is replaced by the discrete modes \(\xi_n=\frac{n\pi}{L}\), \(n\in\mathbb Z\). Thus, the solution for $\mathcal{L}=0$ to \eqref{original_half} with $2L-$periodic boundary condition is
\begin{equation}
\label{periodic_mild_solution}
u(x,t)
=
\sum_{n\in\mathbb Z}
\left(
e^{-\epsilon |n|\pi t/L}\widehat u_n(0)
+
\int_0^t e^{-\epsilon |n|\pi (t-s)/L}\widehat f_n(s) ds
\right)e^{i n\pi x/L},
\end{equation}
where $\widehat u_n(t)$ and $\widehat f_n(t)$ are defined via $\widehat \varphi_n(t)=\frac1{2L}\int_{-L}^{L}\varphi(x,t)e^{-in\pi x/L} dx$. In particular, suppose that $u_0$ and $f(\cdot,t)$, originally given on $(0,L)$, are extended to $(-L,L)$ as odd $2L$-periodic functions. Since the multiplier $e^{-\epsilon |n|\pi t/L}$ is even in $n$, the odd subspace is invariant under the evolution. Thus, \eqref{periodic_mild_solution} reduces to a sine-series expansion,
\begin{equation}
\label{classical_sol_0}
u(x,t)
=
\int_0^L u_0(\xi)G(x,\xi,t) d\xi
+
\int_0^t\int_0^L f(\xi,\tau)G(x,\xi,t-\tau) d\xi d\tau,
\end{equation}
with the kernel $G(x,\xi,t)=\frac{2}{L}\sum_{n=1}^{\infty}e^{-\epsilon n\pi t/L}\sin\!\left(\frac{n\pi x}{L}\right)\sin\!\left(\frac{n\pi \xi}{L}\right)$.

\subsection{The inhomogeneous reaction equation with half-diffusion: $\mathcal{L}=\delta\mathcal{I}$} Let $u(x,t)=e^{\delta t}v(x,t)$, then \(v\) satisfies the inhomogeneous equation \[ \partial_t v=-\epsilon(-\Delta)^\frac{1}{2}v+e^{-\delta t}f(x,t), \qquad (x,t)\in\mathbb R\times(0,T], \] with $v(x,0)=u(x,0)=u_0(x).$ Proceeding as in Section B.1, we obtain \[ \widehat u(\xi,t) =e^{\delta t}\widehat v(\xi,t) =e^{(\delta-\epsilon |\xi|)t}\widehat u_0(\xi) +\int_0^t e^{(\delta-\epsilon |\xi|)(t-s)}\widehat f(\xi,s) ds. \] Therefore, the solution to \eqref{diffusion_u} is 
\begin{equation} \label{mild_delta} 
u(x,t) =\frac1{2\pi}\int_{\mathbb R}e^{i\xi x}e^{(\delta-\epsilon |\xi|)t}\widehat u_0(\xi) d\xi +\frac1{2\pi}\int_{\mathbb R}\int_0^t e^{i\xi x}e^{(\delta-\epsilon |\xi|)(t-s)}\widehat f(\xi,s) ds d\xi. 
\end{equation}

The solution for $\mathcal{L}=\delta \mathcal{I}$ to \eqref{original_half} with $2L-$periodic boundary condition is
\begin{equation}
\label{periodic_mild_delta}
u(x,t)
=
\sum_{n\in\mathbb Z}
\left(
e^{(\delta-\epsilon |n|\pi/L)t}\widehat u_n(0)
+
\int_0^t e^{(\delta-\epsilon |n|\pi/L)(t-s)}\widehat f_n(s)\,ds
\right)e^{in\pi x/L}.
\end{equation}
\subsection{The inhomogeneous advection equation with half-diffusion: $\mathcal{L}=\delta\partial_x$}

Since the operator \((-\Delta)^{\frac12}\) is translation-invariant on $\RR$, we consider solutions of the form \( w(x,t) = u(x-\delta t,t) \) to eliminate the advection. By the chain rule,
\[
\partial_t w(x,t)=\partial_t u(x-\delta t,t)-\delta \partial_x u(x-\delta t,t).
\]
Substituting this into
\[
\partial_t u(x,t)=-\epsilon(-\Delta)^\frac12 u(x,t)+\delta \partial_x u(x,t)+f(x,t),
\qquad (x,t)\in\mathbb R\times(0,T],
\]
we obtain the inhomogeneous equation
\[
\partial_t w(x,t)
=-\epsilon(-\Delta)^\frac12 w(x,t)+f(x-\delta t,t),
\]
with $w(x,0)=u(x,0)=u_0(x).$
Again proceeding as in Section B.1, we obtain
\begin{equation}
\label{mild_advection}
u(x,t)
=\frac1{2\pi}\int_{\mathbb R}e^{i\xi x}e^{(i\xi\delta-\epsilon |\xi|)t}\widehat u_0(\xi) d\xi
+\frac1{2\pi}\int_{\mathbb R}\int_0^t
e^{i\xi x}e^{(i\xi\delta-\epsilon |\xi|)(t-s)}\widehat f(\xi,s) ds d\xi.
\end{equation}

The solution for $\mathcal{L}=\delta \partial_x$ to \eqref{original_half} with $2L-$periodic boundary condition is
\begin{equation}
\label{periodic_mild_advection}
u(x,t)
=
\sum_{n\in\mathbb Z}
\left(
e^{\left(i\delta n\pi/L-\epsilon |n|\pi/L\right)t}\widehat u_n(0)
+
\int_0^t
e^{\left(i\delta n\pi/L-\epsilon |n|\pi/L\right)(t-s)}
\widehat f_n(s)\,ds
\right)e^{in\pi x/L}.
\end{equation}

\section{Practical solutions of Eq. (\ref{sec4_19}) }
\label{appendixC}
We first describe the solution of the auxiliary systems in the general case. For simplicity, we explain how to solve one of the auxiliary systems in \eqref{sec4_19}, omitting the subscript $j$:
\begin{equation}
\begin{bmatrix}
\lambda^\omega I  &   -I \\
   -P      & \lambda^\omega I - Q
\end{bmatrix}
\begin{bmatrix}
{\bm v}^{(1)}_2  \\
{\bm v}^{(2)}_2
\end{bmatrix}
=
\begin{bmatrix}
{\bm v}^{(1)}_1\\
{\bm v}^{(2)}_1
\end{bmatrix}.
\end{equation}
Eliminating ${\bm v}^{(2)}_2$ from the first block equation gives
\begin{equation}
{\bm v}^{(2)}_2 = \lambda^\omega {\bm v}^{(1)}_2 - {\bm v}^{(1)}_1.
\label{eqB2_new}
\end{equation}
Substituting \eqref{eqB2_new} into the second block equation yields
\[
\left[\lambda^\omega(\lambda^\omega I - Q) - P\right]{\bm v}^{(1)}_2
=
{\bm v}^{(2)}_1 + (\lambda^\omega I - Q){\bm v}^{(1)}_1.
\]
Thus one only needs to solve a reduced linear system of half the original size. In the nonperiodic case, this reduced system can be solved efficiently by sparse direct solvers, for example, MUMPS. Once ${\bm v}^{(1)}_2$ is obtained, ${\bm v}^{(2)}_2$ follows from \eqref{eqB2_new}, and hence
\[
\widetilde{{\bm v}}_2 =
\begin{bmatrix}
{\bm v}^{(1)}_2\\
{\bm v}^{(2)}_2
\end{bmatrix}
\]
is recovered.

\paragraph{Accelerated implementation in the case $PQ = QP$}
Since $P,Q$ are simultaneously diagonalizable, then there exists an orthogonal transform matrix $U$ such that
\[
P = U^{*}\Lambda_P U,
\qquad
Q = U^{*}\Lambda_Q U,
\]
where $\Lambda_P=\operatorname{diag}(p_1,\ldots,p_m)$ and $\Lambda_Q=\operatorname{diag}(q_1,\ldots,q_m)$ are diagonal. Then the block matrix $\mathbf{D}$ is block-diagonalized by $U$ into $m$ independent $2\times 2$ blocks
\[
\widehat{\mathbf{D}}_\ell=
\begin{bmatrix}
0 & 1\\
p_\ell & q_\ell
\end{bmatrix},
\qquad \ell=1,\ldots,m.
\]
Therefore, each system in \eqref{sec4_19},
\[
(\lambda_j^\omega I_m-\tau \mathbf{D})\widetilde{\bm v}_{2,j}
=
\widetilde{\bm v}_{1,j},
\]
can be further decoupled into $m$ independent $2\times 2$ systems
\[
(\lambda_j^\omega I_2-\tau \widehat{\mathbf{D}}_\ell)\widehat{\bm v}_{2,j,\ell}
=
\widehat{\bm v}_{1,j,\ell},
\qquad \ell=1,\ldots,m.
\]
Since
\[
\lambda_j^\omega I_2-\tau \widehat D_\ell
=
\begin{bmatrix}
\lambda_j^\omega & -\tau\\
-\tau p_\ell & \lambda_j^\omega-\tau q_\ell
\end{bmatrix},
\]
the corresponding solution is obtained explicitly by inverting these $2\times 2$ matrices mode by mode.

Typical examples include periodic boundary conditions, where $U$ is the discrete Fourier transform and the fast algorithm is FFT-based, and homogeneous Dirichlet-type settings, where $U$ is given by a DST. Therefore, the acceleration mechanism depends only on the boundary conditions through the transform that simultaneously diagonalizes $P$ and $Q$.

\end{appendices}
\bibliographystyle{IEEEtran}
\bibliography{ref}
\end{document}